\documentclass[11pt]{article}
\usepackage[margin=1.4in]{geometry}
\usepackage{graphicx} 
\usepackage{amsfonts, amssymb, amsthm, amsmath}
\usepackage{mathtools}
\usepackage{bm}
\usepackage{physics} 
\usepackage[hidelinks]{hyperref}
\usepackage{mathtools}
\usepackage{xcolor}
\usepackage{todonotes}
\usepackage{float}
\usepackage{placeins}
\usepackage{etoolbox}
\usepackage{authblk}

\newtheorem{thm}{Theorem}
\newtheorem*{thm*}{Theorem}
\newtheorem{prop}{Proposition}[section]
\newtheorem{lemma}{Lemma}[subsection]

\theoremstyle{definition}

\theoremstyle{remark}

\newcommand{\inner}[2]{\langle {#1}, {#2} \rangle}
\newcommand{\R}{\mathbb{R}}

\newcommand{\T}{\mathbb{T}}

\BeforeBeginEnvironment{thm}{\FloatBarrier}
\AfterEndEnvironment{thm}{\FloatBarrier}

\allowdisplaybreaks

\title{Mixing and enhanced dissipation in a time-translating shear flow}
\author[1]{Johannes Benthaus}
\author[2]{Giuseppe Maria Coclite}
\author[1]{Camilla Nobili}
\affil[1]{School of Mathematics and Physics, University of Surrey} 
\affil[2]{Dipartimento di Meccanica, Matematica e Management, Politecnico di Bari}

\begin{document}

\maketitle

\begin{abstract}
Motivated in part by the work of Vanneste and Byatt-Smith, we study mixing and enhanced dissipation for the advection-diffusion equation with velocity field \[\mathbf{u}(x,y,t)=(\sin(y-ct),0),\] a shear flow whose profile translates rigidly with speed $c$. This is a prototypical example of a flow whose critical points move in time. We quantify how the decay properties depend on the relation between translation speed $c$ and diffusivity $\nu$.
We first analyse the inviscid transport problem and establish time-averaged $H^{-1}$ mixing estimates for $t\lesssim c^{-1}$, yielding decay rates faster than stationary estimates.
Building on these estimates, we prove enhanced dissipation for moderate translation speeds $c=c_0\nu^\ell$ with $\ell\in(1/3,3/4)$. In this regime we obtain decay at rate
$\nu^{(1+2\ell)/5}$, which interpolates continuously between the sharp rates
$\nu^{1/2}$ for stationary shear flows with simple critical points and $\nu^{1/3}$ for monotone flows. This quantifies how increasing translation speed progressively weakens the influence of the critical points. Comparing the inviscid mixing and enhanced dissipation timescales heuristically explains the lower endpoint $\ell=1/3$.
For $c\gg 1$, we show that solutions remain close to those of the heat equation on fixed time intervals, such that the rapid translation averages out advection and weakens mixing.
The mixing estimate relies on a refined stationary phase analysis exploiting cancellations generated by the motion of the critical points. The enhanced dissipation result requires an adaptation of the hypocoercivity framework for stationary shear flows to the non-autonomous setting. The translating flow prevents the commutator hierarchy from closing in the standard way, which we overcome by constructing an extended energy functional. The large-$c$ analysis exploits the averaging effect of the rapid translation in this regime.
\end{abstract}

\newpage

\section{Introduction}
We introduce the passive-scalar problem on the set $\Omega = \mathbb{T}^2 \times [0,\infty)$. Let
the concentration $\Theta = \Theta(x,y,t) : \Omega \to \R$ and incompressible velocity field 
$\bm{u} = \bm{u}(x,y,t) : \Omega \to \R^2$ solve the Cauchy problem for the advection--diffusion equation
\begin{equation}\label{adv}
  \begin{array}{rrl}
    \partial_t \Theta + \bm{u}\cdot\nabla \Theta - \nu \Delta \Theta &=& 0,\\
    \nabla\cdot \bm{u} &=& 0,\\
    \Theta(x,y,0) &=& \Theta_0(x,y),
  \end{array}
\end{equation}
where $\nu \ll 1$ is the molecular diffusivity and $\Theta_0$ is the initial datum, which we assume to be mean free in $x$, i.e.
\begin{align}
\label{mean_free}
\int_\T \Theta_0(x, y) \, dx = 0.
\end{align}
In this paper we investigate enhanced dissipation and the associated mixing mechanisms for a prototypical example within the class of time-dependent shear flows,
\begin{equation}
\label{shear_flow}
  \bm{u}(x,y,t) = (v(y,t), 0)^{\mathsf T}.
\end{equation}

Enhanced dissipation is a physical phenomenon intrinsically connected
with mixing. The mechanism relies on advection generating increasingly fine spatial gradients, 
on which diffusion acts more efficiently, leading to a decay rate faster than that produced by diffusion alone.
This interplay between advection and diffusion has been widely studied
in the fluid-dynamics literature, and we therefore restrict ourselves
here to highlighting a few contributions most relevant to our work.
Batchelor \cite{batchelor1959} developed the foundational theory of
small-scale scalar fluctuations. Pierrehumbert
\cite{pierrehumbert1994tracer} subsequently identified, through numerical
simulations of alternating shear maps, what he termed \emph{strange
eigenmodes}: persistent spatial patterns that govern the long-time decay
of the scalar variance. The spectral mechanisms underlying such behaviour
were further investigated by Reddy \& Trefethen \cite{reddy1994}, who
emphasised the role of non-normality (arising from the
non-self-adjointness of the advection--diffusion operator) in transient
growth, and by Giona et al.~\cite{giona2004}, who analysed the
eigenvalue structure for steady flows. The existence of strange
eigenmodes was later established rigorously by Liu \& Haller
\cite{liu2004} using inertial manifold theory. For a more complete
overview we refer to the survey \cite{thiffeault2012} and references
therein.

The rigorous mathematical treatment of enhanced dissipation with the aim of obtaining quantitative decay rates for the whole time-evolution is more
recent.  For autonomous shear flows, the theory is by now
well-established, and we refer to the surveys \cite{cotizelati2024,mazzucato2026} for
a comprehensive overview. 
A key result, first established by Coti Zelati \& Bedrossian
\cite{Bedrossian2017} and later extended by Coti Zelati \& Gallay
\cite{cotizelatigallay}, shows that the enhanced dissipation timescale is
$\sim\nu^{m/(m+2)}$, which is known to be sharp
\cite{Zelati2021}. Here $m$ denotes the maximal order of degeneracy of
the critical points of the shear profile, the points where the velocity
gradient vanishes and the shear can no longer generate fine-scale
structure. The degeneracy of these critical points is thus the
determining factor for the enhanced dissipation rate.

Although physically relevant, the non-autonomous setting, that is when
$\mathbf{u}$ depends on time, has been explored far less, largely due
to the additional analytical difficulties it entails. We mention in
passing the related results in stochastic settings
\cite{Bedrossian2021,cooperman2025,Gess2025,Seis2025}, where the time
dependence arises through the randomness of the velocity field. In the
deterministic setting, Coble \& He \cite{Coble2024} considered a
perturbative regime in which the rate of change of the shear profile is
$\mathcal{O}(\nu^{3/4})$ and showed that the enhanced
dissipation rate of the stationary problem is preserved. In previous work \cite{benthaus2026}, some of the present authors studied the case of a separable shear flow profile $v(y,t) = \xi(t)\, w(y)$ and showed that the enhanced
dissipation timescale depends directly on the time integral of the
modulation function~$\xi(t)$, and need not match the stationary rate. In
particular, suitable choices of $\xi(t)$ can produce dissipation rates
that are strictly faster than those associated with the corresponding
stationary flow. Further, Elgindi et al.~\cite{elgindi2025} established
enhanced dissipation for a flow that alternates between two stationary
profiles. Notably, in each of these settings the critical points of the velocity profile either remain fixed in time or do not play a direct role in the mechanism driving the improved decay. 
While there exist constructions of time-dependent flows with moving
critical points that yield faster decay rates~\cite{he2025}, the
precise dependence of the enhanced dissipation rate on the translation
speed has not been determined.
In contrast, in the present work the motion of the critical points induced by the time translation is precisely the mechanism generating the enhanced mixing and dissipation, and we provide a quantitative description of this effect.

\medskip
In this paper we are interested in the vector field \eqref{shear_flow} with
\begin{equation}
\label{sine_flow}
  v(y,t) = \alpha \sin (y - c t), \qquad \alpha, c \in \R_+.
\end{equation}

The motivation for studying this flow stems from the work of Vanneste \& Byatt-Smith \cite{Vanneste2007},
 where asymptotic and numerical analysis show that a translating shear flow with 
 speed $c\sim\nu^{1/2}$ enters a distinguished regime in which the principal 
 eigenvalue retains an $\mathcal{O}(1)$ real part as $\nu\to0$, formally yielding an 
 $\mathcal{O}(1)$ long-time decay rate. 
Crucially, however, the same study demonstrates that this fast decay is not robust. 
First, it governs only the ultimate asymptotic regime and is preceded by a much 
slower transient of duration $\sim\nu^{-1/2}$.
Second, the advection--diffusion 
operator is strongly non-normal, meaning that its eigenfunctions are highly non-orthogonal 
and small perturbations---such as numerical round-off or weak physical noise---can excite
 transient structures that are not true eigenmodes. These so-called \emph{pseudomodes}, 
 which are approximate eigenfunctions with much smaller decay rates, dominate the dynamics 
 and reduce the effective decay to $\sim\nu^{2/5}$.

Taken together, these observations reveal a significant gap between the asymptotic and numerical
 picture on the one hand and the existing rigorous quantitative theory on the other, 
 thereby motivating several questions. 
We begin with the inviscid problem, asking how much stronger
the mixing mechanism becomes when the critical points translate 
through the domain. This is motivated by the connection
between mixing and enhanced dissipation established by Coti Zelati
\emph{et al.}~\cite{cotizelati2019}, which shows that quantitative
mixing in the inviscid problem implies enhanced dissipation in the
viscous setting. As we shall see, a comparison between the inviscid mixing timescale and the enhanced dissipation timescale heuristically explains the range of translation speeds for which decay beyond the stationary rate can occur.
 We then turn to the viscous dynamics at moderate translation speeds, asking in particular whether 
 enhanced dissipation---or even decay faster than in the stationary setting---can persist beyond 
 the perturbative regime suggested by the asymptotic and spectral analysis of \cite{Vanneste2007},
  namely for translation speeds larger than $c \gtrsim \nu^{3/4}$. Finally, we consider the opposite 
  limit of very large translation speeds $c \gg 1$, where mixing is expected to deteriorate, 
  and investigate what decay properties, if any, remain in this regime.

The present paper provides quantitative results addressing each of these regimes. 
We establish three main theorems: the first concerns inviscid mixing, while the 
second and third quantify viscous decay in two distinct translation regimes.

Before presenting our new results, we recall that in problems of mixing and enhanced dissipation,
sharp estimates are often obtained in Fourier variables, where transport redistributes energy 
across modes while diffusion acts directly on frequencies. 
For this reason, two of our main results are formulated and proved after taking the Fourier 
transform in the horizontal direction.

Since the shear depends only on $y$, we expand $\Theta$ in Fourier modes in $x$, reducing the 
problem to a family of one–dimensional equations in $y$, one for each wavenumber $k\in\mathbb{Z}$.
 Our analysis is therefore carried out mode by mode. Writing $\widehat{\Theta}(k,y,t)$ for the
  $k$-th Fourier coefficient, the equation becomes
\begin{equation}
\label{adv_fourier}
  \partial_t \widehat{\Theta}(k,y,t) + i \alpha k \sin(y-ct) \widehat{\Theta}(k,y,t) + \nu k^2 \widehat{\Theta}(k,y,t) - \nu \partial_y^2 \widehat{\Theta}(k,y,t) = 0.
\end{equation}

By contrast, the third result concerns long-time behaviour in the rapidly translating regime, where the dynamics becomes diffusion-dominated. In that setting the Fourier decomposition is no longer essential, and the analysis is carried out directly in physical space.

 \paragraph{Time-averaged inviscid mixing.}

We aim to establish a quantitative, time-averaged mixing estimate for the inviscid 
transport problem ($\nu=0$) in \eqref{adv_fourier}. 
Since the profile $v(y,t)=\sin(y-ct)$ is periodic in time with period $2\pi/c$, 
solutions of the associated transport equation inherit this periodicity and therefore 
cannot exhibit mixing at a uniform rate for all times. Any mixing mechanism must
 therefore act on a bounded timescale. 
In particular, if the translation is to produce enhanced dissipation beyond the 
stationary rate, this mechanism must be quantitatively stronger than in the autonomous case.
\begin{figure}[H]
    \centering
    \includegraphics[width=0.9\textwidth]{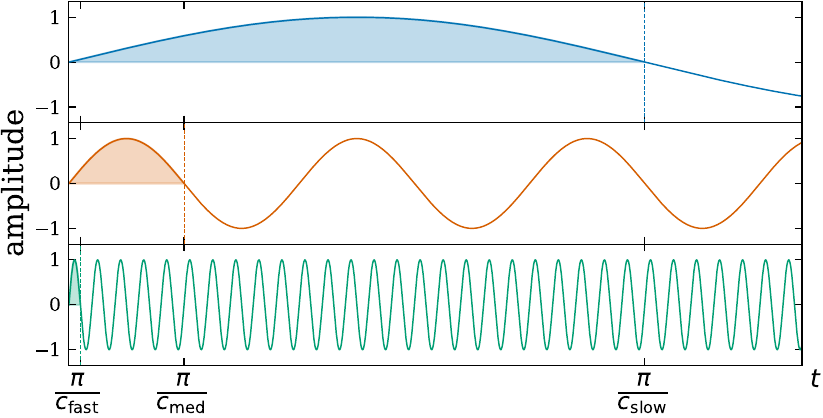}
    \caption{
    Illustration of the time–dependent shear $t \mapsto \sin(y-ct)$ at fixed spatial location $y$ for three different $c$. 
    The shaded regions corresponds to the first sign–coherent intervals
    during which phase gradients accumulate and mixing strengthens. 
    Once the shear changes sign, the dynamics begins to unwind previously generated gradients, 
    explaining why the mixing estimate is naturally restricted to times of order $c^{-1}$.
    }
    \label{fig:mixing_window}
\end{figure}
Motivated by this constraint, we introduce a time-averaged $H^{-1}$ mixing measure
    \begin{align}
    \label{mix:M:H_minus}
\biggl\|
    \frac{1}{T}\int_1^T \widehat{\Theta}(k,\cdot,t)\,dt
    \biggr\|_{H_y^{-1}}
:=\sup_{\eta\in H^1_y,\;\norm{\eta}_{H_{y}^1}=1}\abs{
    \frac{1}{T}\int_1^T\int_{\T}\widehat{\Theta}(k,y,t)\overline{\eta(y)}\,dy\, dt}.
\end{align}
and prove the following theorem for the Fourier coefficients 
    $\widehat{\Theta}(k,y,t)$ of $\Theta$:

    \begin{thm}
  \label{mix:thm:time-avg-mixing}
    Let $0<c\leq 1$ and $\widehat{\Theta}$ solve \eqref{adv_fourier} with $\nu=0$ and
    initial datum $\widehat{\Theta}_0\in H^1_y(\T)$. 
    Then there exists a
    constant $C>0$, independent of $c$ and $T$, such that for every
    $k\in\mathbb Z\setminus\{0\}$ and every $T\in (1,\frac{\pi}{c}]$,
    \begin{equation}\label{int:eq:time-avg-prop}
    \biggl\|
    \frac{1}{T}\int_1^T \widehat{\Theta}(k,\cdot,t)\,dt
    \biggr\|_{H_y^{-1}}
     \;\le\;
     C\,\frac{1}{T}
     \biggl(\frac{(\ln T)^2}{c\,|\alpha k|^2}\biggr)^{\!1/3}
     \,\|\widehat{\Theta}_0(k,\cdot)\|_{H_y^{1}}.
\end{equation}
\end{thm}
  
  Estimate \eqref{int:eq:time-avg-prop} controls $H^{-1}-$norm of the \textit{time-averaged solution}, thereby capturing cancellations generated by the temporal oscillation of the shear. In contrast, stationary mixing estimates only control the time average of instantaneous norms. Since \[ \biggl\|
    \frac{1}{T}\int_1^T \widehat{\Theta}(k,\cdot,t)\,dt
    \biggr\|_{H_y^{-1}}\leq \frac{1}{T}\int_1^T\norm{\widehat{\Theta}(k,\cdot,t)}_{H^{-1}_y} dt \]
    such stationary estimates (expoiting that the shear is bounded in $x$ and $t$) would yield at best a $T^{-\frac 12}$ decay. The moving critical points of the shear therefore produce a genuinely stronger effect, leading to the almost 
    $T^{-1}$ decay (up to logarithms) in \eqref{int:eq:time-avg-prop}.

\medskip
   Finally, we remark that the mixing mechanism exploited here operates while the shear preserves a coherent sign.
For a fixed spatial location $y$, the shear $t \mapsto \sin(y-ct)$ maintains a constant sign only over time intervals of length $\pi/c$. During such intervals, phase gradients accumulate and mixing strengthens. Once the shear changes sign, the transport begins to unwind these gradients, so the mixing is no longer monotone. This explains why the estimate is naturally restricted to times $T \lesssim c^{-1}$.

    We prove this result via the method of stationary phase applied in both space and
    time, where a crucial step is the construction of a suitable space-time partition
    isolating the moving critical points.

    \begin{figure}
  \centering
  \includegraphics[width=0.9\textwidth]{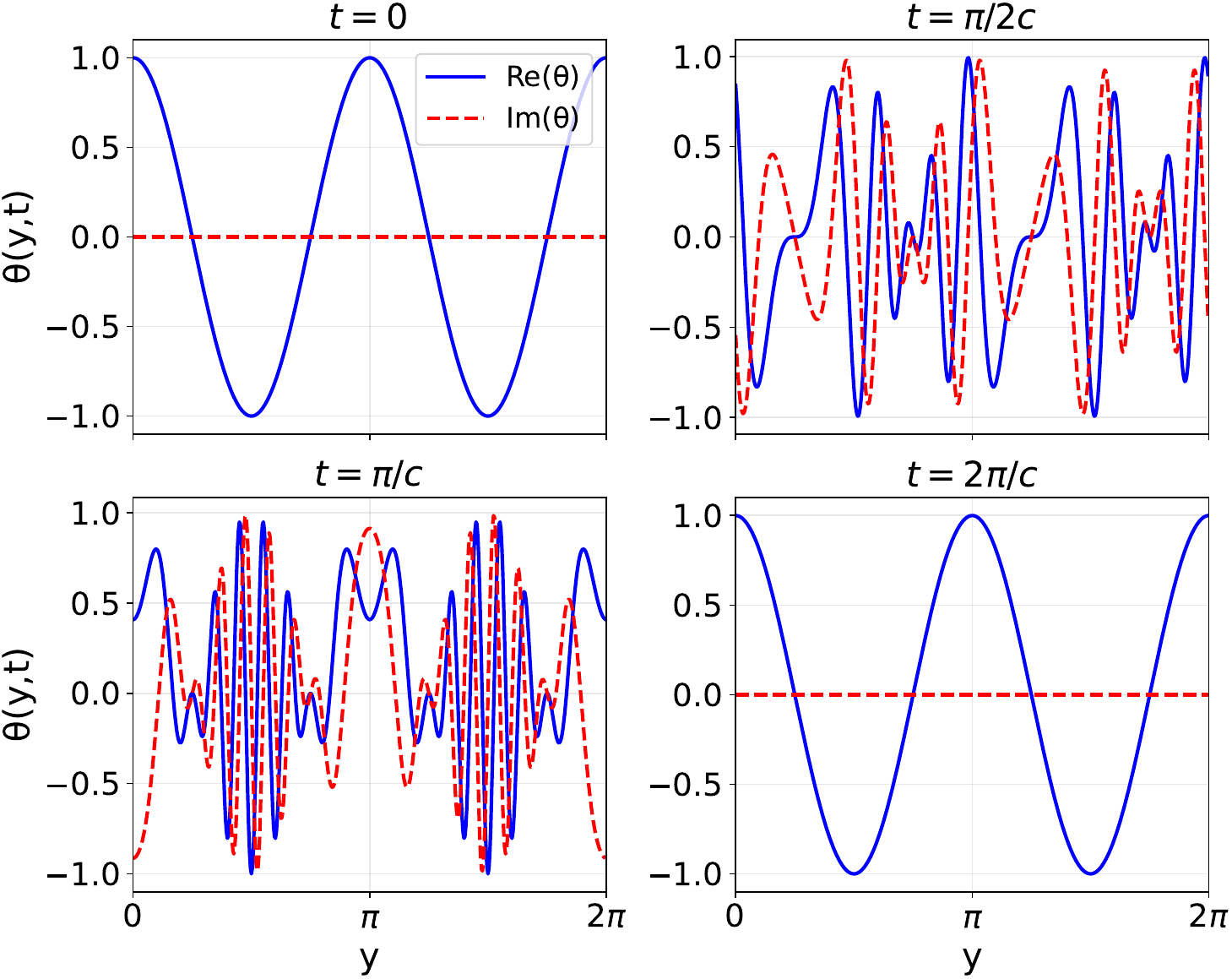}
  \caption{Snapshots of the inviscid solutions 
$\hat\Theta(k,y,t)$ first Fourier mode ($|k|=1$), 
with initial datum $\hat\Theta_0 = \cos(2y)$ and translation 
speed $c = 0.1$, at times 
$t = 0,\, \pi/2c,\, \pi/c,\, 2\pi/c$. 
The solution develops increasingly fine-scale oscillations 
within the time interval $[0, \pi/c]$, 
consistent with the mixing mechanism of 
Theorem~\ref{mix:thm:time-avg-mixing}. After one full period 
($t = 2\pi/c$), the solution returns to its initial profile.}
\label{fig:inviscid_snapshots}
\end{figure}
\FloatBarrier
    \paragraph{Intermediate translation speeds.}

    In the presence of viscosity, i.e.\ when $\nu\neq 0$, and for intermediate translation speeds, the analysis becomes more involved. The main difficulty lies in capturing the time dependence of the shear and quantifying the influence of the translation speed on the decay rate. We address this issue by adapting the hypocoercivity framework introduced by Villani~\cite{villani}. We prove

    \begin{thm}\label{hyp:thm:L2_phys}
    Let $\ell\in\big(\tfrac13,\tfrac34\big)$, $\beta_0\in(0,1)$ and assume $c = c_0\,\nu^\ell$. There exists $\mu_0\in(0,1)$ such that for all $k\in\mathbb Z\setminus\{0\}$ and $0<\nu\le \mu_0\,\alpha|k|$
    the solution $\widehat{\Theta}(k,y,t)$ to \eqref{adv_fourier}
    with initial datum $\widehat{\Theta}_0\in L^2(\T)$ satisfies
    \begin{equation}\label{hyp:L2_phys}
    \begin{aligned}
  \|\widehat{\Theta}(k,\cdot,t)\|_2^2
  \;\le\;&\;
  C_{\mathrm{ed}}\Bigl(1+\beta_0^{1/2}\Bigl(\frac{\alpha|k|}{\nu}\Bigr)^{\frac{1+2\ell}{5}}\Bigr)\times
  \\
  &\exp\!\Bigl(
    -\frac{\beta_0^{5/4}}{12C_s}\,c_0\,(\alpha|k|)^{\frac{3\ell-1}{5}}\,\nu^{\frac{1+2\ell}{5}}\,t
    -2\nu k^2 t
  \Bigr)\,
  \|\widehat{\Theta}_0(k,\cdot)\|_2^2,
    \end{aligned}
    \end{equation}
 $\forall t\ge0$. Here, $C_{\mathrm{ed}}>0$ and $C_s\ge 1$ are constants independent of $\nu, k$.
    \end{thm}
We first highlight several features of the decay rate in \eqref{hyp:L2_phys}.
In the admissible range of~$\ell$, the viscosity exponent is
$\tfrac{1+2\ell}{5}$, and therefore represents a linear interpolation between
the endpoints $\ell = \tfrac{1}{3}$, yielding a rate $\sim\nu^{1/3}$, and
$\ell = \tfrac{3}{4}$, yielding a rate $\sim\nu^{1/2}$. Both endpoints admit a natural interpretation in terms of the known sharp rates
for stationary shear flows. The enhanced dissipation rate for a stationary flow
with simple critical points (i.e.\ $c=0$ in~\eqref{sine_flow}) is known to
scale as $\nu^{1/2}$ (cf.~\cite{Bedrossian2017}), which is
sharp~\cite{Zelati2021}. Therefore, throughout the admissible range the rate is
strictly faster than in the stationary setting, reflecting the effect of moving
critical points. The upper endpoint $\ell = \frac{3}{4}$ coincides with the
endpoint of the perturbative regime identified in~\cite{Coble2024}, and the corresponding rates agree at this threshold. Our estimate therefore connects seamlessly to the
perturbative regime. At the lower endpoint $\ell \to \frac{1}{3}$, the rate approaches
$\nu^{1/3}$, which is the sharp rate for stationary monotone shear
flows~\cite{Zelati2021}. This suggests that sufficiently rapid
translation effectively weaken the degeneracy of the critical points.
Heuristically, the lower endpoint $\ell=1/3$ can be understood by comparing
the natural timescales of mixing and diffusion. Since $c\sim\nu^\ell$,
the shear profile $\sin(y-ct)$ (fixed $y$) crosses the axis on the timescale
$t\sim 1/c\sim\nu^{-\ell}$. This is also the maximal time window on which
the inviscid mixing estimate of Theorem~1 remains effective.

On the other hand, the enhanced dissipation rate obtained in
Theorem~2 corresponds to a decay timescale
$t\sim\nu^{-(1+2\ell)/5}$. Equating the two timescales
\[
\nu^{-(1+2\ell)/5}\sim\nu^{-\ell}
\]
yields $\ell=1/3$. This suggests that when $\ell<1/3$ the shear
reverses direction before sufficient gradients have accumulated to sustain
enhanced dissipation at this rate, so that the mixing mechanism can no
longer drive the decay at the scale $\nu^{(1+2\ell)/5}$. While this argument is purely heuristic, it
provides a natural explanation for the endpoint $\ell=1/3$ appearing
in our results. 

Moreover, numerical
computations for a range of~$\nu$ values confirm that our scaling agrees with
the numerically recovered rate with high confidence
(see Figure~\ref{fig:sweep}). The acceleration of decay compared to the
stationary setting is also consistent with the asymptotic eigenvalue analysis
of~\cite{Vanneste2007}.

In contrast to the stationary flow and the perturbative regime, the scaling in
the wavenumber~$k$ does differ from the scaling in~$\nu$. Whether this is a
genuine feature of the intermediate regime or an artifact of our method remains
an interesting open question. Additionally, the estimate~\eqref{hyp:L2_phys} carries a
logarithmic prefactor in front of the exponential. In the stationary setting,
such corrections can be removed via time-dependent weights in the hypocoercivity
method (cf.~\cite{Wei2019}). It remains open whether an analogous approach
applies in the present setting.

We now turn to the key analytical difficulty underlying
Theorem~\ref{hyp:thm:L2_phys}. In contrast to the stationary setting (cf.~\cite{Bedrossian2017}) or to time-modulated flows (cf.~\cite{benthaus2026}), the time dependence of the shear prevents the hypocoercive structure from closing at finite order. The commutator chain no longer terminates but instead produces a cyclic interaction between cosine- and sine-weighted components generated by the translation of the profile. Geometrically, this reflects the fact that the motion of the critical points continually redistributes spatial oscillations rather than allowing them to accumulate in a fixed direction. To capture this effect, we extend the energy functional by
introducing additional commutator levels, which allow us to track the transfer
of energy across these oscillating modes. The resulting estimate shows that when
the translation speed scales as $c \sim \nu^\ell$ with
$\ell \in (1/3, 3/4)$, the time dependence of the shear
produces a genuine mixing mechanism that enhances dissipation beyond the purely
diffusive scale.

\paragraph{Fast translation speeds.}
  In the regime of large $c$, we prove that solutions of
\eqref{adv} remain close in $L^2$ to the solution
$\Theta_{\mathrm{H}}$ of the heat equation with the same
initial datum, on any fixed time interval, with an error
that improves as $c$ increases.
    \begin{thm}\label{lc:thm:large_c}
    Let  $\Theta$ solve \eqref{adv} with initial datum $\Theta_0\in H^2(\T^2)$. For every $T_*>0$, $0 < \nu \le 1$ and $c > \frac{5}{2}$, the following estimate holds for all $t\in[0,T_*]$:
    \begin{equation}\label{int:eq:diff-estimate}
     \|\Theta(\cdot,t)-\Theta_{\rm{H}}(\cdot,t)\|_2^2
    \le C
   \Bigl(
    \frac{1}{c}+\frac{\nu}{c}\Bigr)
    \Bigl(\|\nabla^2 \Theta_0\|^2_2+\|\partial_x \Theta_0\|^2_2\Bigr)
    e^{C t},
    \end{equation}
    where $C>0$ is a constant  independent of $\nu, c.$ 
    \end{thm} 
For rapidly translating profiles, the shear oscillates on a time scale much 
shorter than the diffusive one. In particular, the mixing time-window 
$t\sim c^{-1}$ from Theorem~\ref{mix:thm:time-avg-mixing} shrinks to zero, so 
the mixing mechanism that drives enhanced dissipation in the intermediate 
regime can no longer act. As a result, the transport averages out over 
finite time horizons, and no enhanced dissipation mechanism can develop. This is reflected in the absence of any improved decay rate in 
\eqref{int:eq:diff-estimate}, as well as in the exponential factor $e^{Ct}$, 
which arises from a Gr\"onwall argument and limits the estimate to fixed 
time intervals.

Nevertheless, the estimate shows that the effect of advection on the energy remains small when the translation speed $c$. Indeed, in this case the prefactor $\tfrac{1}{c}$ implies that the solution of the advection–diffusion equation stays close, in $L^2$, to the solution of the corresponding heat equation on any fixed time interval. In this sense, rapidly oscillating transport acts only as a weak perturbation of diffusion, and the dynamics becomes asymptotically diffusion–dominated as $c\to\infty$.

\medskip
The remainder of this paper is organised as follows. In
Section~\ref{sec:mix} we establish the time-averaged mixing estimate for
the inviscid problem (Theorem \ref{mix:thm:time-avg-mixing}), using a
space-time stationary phase argument. Section~\ref{sec:hyp} contains the
proof of the enhanced dissipation result at intermediate translation speeds
(Theorem \ref{hyp:thm:L2_decay}). We first introduce the extended
hypocoercive energy functional and prove exponential decay of the functional
(Proposition \ref{hyp:prop:hypo_decay}). The $L^2$-estimate is then
recovered by combining this decay with a gradient bound obtained from the
energy balance over an initial transient layer. In Section~\ref{sec:lc} we
treat the large translation speed regime and prove
Theorem~\ref{lc:thm:large_c}. An integration by parts in time extracts the
smallness $c^{-1}$ from the rapidly oscillating advection. Closing the
resulting estimate requires controlling a system of higher-order energy balances and a Gr\"onwall argument. The Appendix collects auxiliary lemmas, detailed estimates
deferred from the main text, and the parameter optimisation for the
hypocoercive functional.

\begin{figure}
\centering
\includegraphics[width=0.85\textwidth]{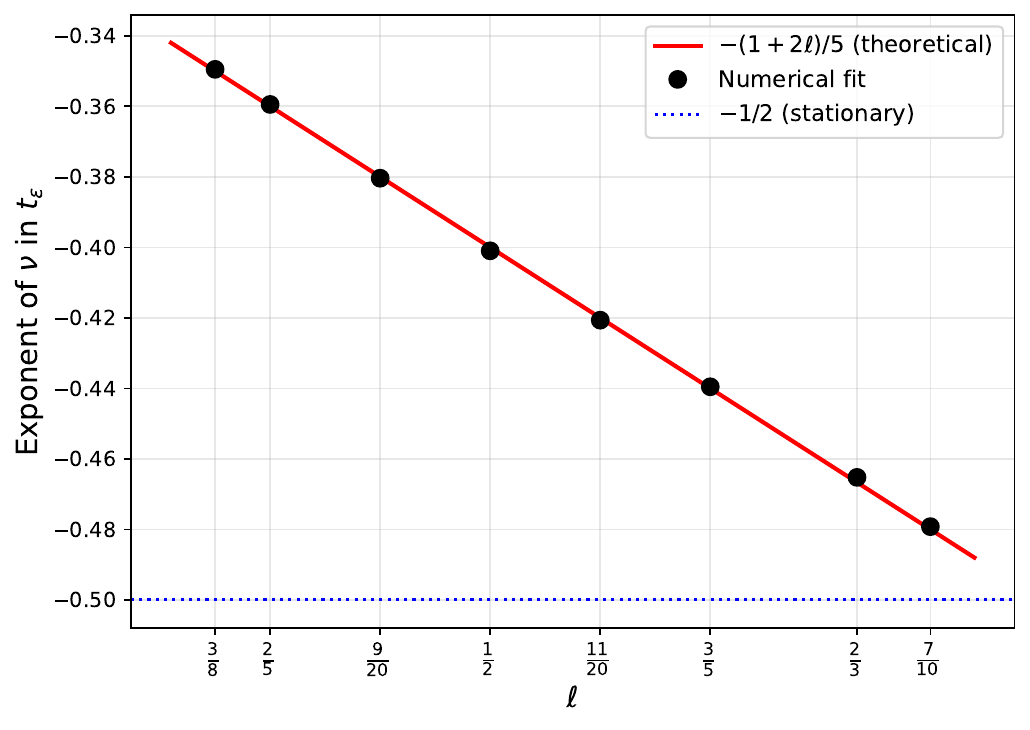}
\caption{Numerically fitted exponent of $\nu$ in the dissipation time
$t_\varepsilon$ (defined as the first time the norm drops below $\varepsilon = 0.1$) as a
function of~$\ell$, where $c = c_0\,\nu^\ell$. Exponents are extracted
via log--log regression of $t_\varepsilon$ against~$\nu$ over ten values
logarithmically spaced in $\nu \in [3\times 10^{-7}, 10^{-2}]$. The
solid line is the theoretical prediction $-(1+2\ell)/5$ from
Theorem~\ref{hyp:thm:L2_phys}, the dotted line the corresponding stationary exponent
$-1/2$. The fitted exponents agree with the predicted curve to within an average relative error of
$0.17\%$.}
\label{fig:sweep}
\end{figure}
\medskip
In the following, $C$ denotes a constant that is independent of the crucial system
parameters (like $c,k,\nu$ etc.) but might change value from line to line. On occasion
we write $A\lesssim B$ if $A\leq CB$ for such a constant $C$, and $A\sim B$ if
$A\lesssim B$ and $B\lesssim A$.

We denote by $\|\cdot\|_2$ the $L^2$-norm, by $\|\cdot\|_{H^s}$ the standard
Sobolev norms, and by $\langle\cdot,\cdot\rangle$ the $L^2$ inner product, taken
conjugate-linear in the second argument. Unless otherwise indicated, these refer to
$L^2(\T^2)$ in physical space and to $L^2_y(\T)$ after taking the Fourier
transform in~$x$.
\FloatBarrier
\section*{Acknowledgments}

J.B. warmly thanks the Department of Mechanics, Mathematics and Management of the Politecnico di Bari for the hospitality, and Dominic Stone for insightful discussions.

G. M. Coclite has been partially supported by the Project  funded  under  the  National  Recovery  and  Resilience  Plan  (NRRP),  Mission  4, Component  2,  Investment  1.4 (Call  for  tender  No.  3138  of  16/12/2021), of  Italian  Ministry  of  University and Research funded by the European Union (NextGenerationEU Award,  No.  CN000023,  Concession  Decree  No.  1033  of  17/06/2022)  adopted  by  the  Italian Ministry of University and Research (CUP D93C22000410001), Centro Nazionale per la Mobilit\`a Sostenibile.  He has also been supported by the 
Italian Ministry of Education, University and Research under the  Programme ``Department of Excellence'' Legge 232/2016 (CUP D93C23000100001), and by the Research Project of National Relevance ``Evolution problems involving interacting scales'' granted by the Italian Ministry of Education, University and Research (MIUR PRIN 2022, project code 2022M9BKBC, CUP D53D23005880006). He warmly thanks the Department of Mathematics of the University of Surrey for the hospitality.

C.N. was partially supported by the LMS Scheme 4 Grant (Ref. No.~4243642436). C.N. thanks Steffen Pottel for helpful comments and insightful discussions on mixing.

\bigskip
\section{Mixing estimates for the inviscid problem}
\label{sec:mix}

In this section we prove Theorem~\ref{mix:thm:time-avg-mixing}.
To this end, we consider the inviscid counterpart of~\eqref{adv}, obtained by setting $\nu=0$. This yields the transport equation
\begin{equation}\label{mix:transport}
\partial_t \Theta(x,y,t) + \bm{u}\cdot \nabla \Theta(x,y,t)=0 .
\end{equation}
We recall from \eqref{shear_flow}–\eqref{sine_flow} that the velocity field is given by
\[
\bm{u}(x,y,t) = (v(y,t), 0)^{\mathsf T},
\qquad \text{where} \qquad
v(y,t) = \alpha \sin (y - c t).
\]
Taking the Fourier transform of \eqref{mix:transport} in the
$x$-variable decouples the Fourier modes $k\in\mathbb{Z}$.
Consequently, the Fourier coefficients
$\widehat{\Theta}(k,y,t)$ evolve independently according to
the mode-by-mode equation
\begin{equation}
\label{mix:transport_four}
  \partial_t \widehat{\Theta}(k,y,t)
  + + i \alpha k\, \sin(y-ct)\, \widehat{\Theta}(k,y,t) = 0. 
\end{equation}
Since $\alpha$ enters only through the product $\alpha$, we set $\alpha=1$ throughout this section; the factor $\alpha$ is restored in the final estimate. The solution then reads
\begin{align*}
  \psi(y,t) := \int_0^t \sin(y - c\tau)\,d\tau, \qquad
\widehat{\Theta}(k,y,t) = \widehat{\Theta}_0(k,y)\,e^{-ik\,\psi(y,t)}.
\end{align*}

\subsection{Proof of Theorem \ref{mix:thm:time-avg-mixing}}
We recall that for a fixed wave number $k\in\mathbb Z\setminus\{0\}$, we have solutions to \eqref{mix:transport_four} that can be expressed as
\[
   \widehat{\Theta}(k,y,t)=\widehat{\Theta}_{0}(k,y)\,e^{-ik \psi(y,t)},
\]
where
    \begin{align}
    \psi(y,t) := \int_0^t \sin(y - c\tau)\,d\tau=\frac{\cos(y-ct)-\cos(y)}{c}.
\end{align}
Let $\eta\in H^1_y(\mathbb T)$ with $\|\eta\|_{H^1_y}=1$ and set
\[
   \phi(y):=\widehat{\Theta}_{0}(k,y)\,\overline{\eta(y)}.
\]
Then the goal is to find an upper bound for the following quantity
\begin{align}\label{mix:qty}
 \frac{1}{T}\int_{t_0}^T \int_{\mathbb T}
         e^{-ik \psi(y,t)}\phi(y)\,dy\, dt,
\end{align}
where  $t_0$ will be determined later.\\
We estimate \eqref{mix:qty} by a non-stationary phase argument, based on integration by parts in $y$ and $t$.
Away from the region where $|\partial_y\psi|$ or $|\partial_t\psi|$ is small we integrate by parts in the corresponding variable,
while the remaining ``critical'' region is controlled by a measure bound.
In particular, we first identify the sets on which integration by parts is effective, and then define the critical sets accordingly.

\paragraph{Step 1: Critical sets.}
We start by computing derivatives as
\begin{align*}
    \partial_y\psi(y,t)=\frac{\sin(y)-\sin(y-ct)}{c}= \frac{2}{c}\sin(\frac{ct}{2})\cos(\frac{ct}{2}-y)
\end{align*}
and
\begin{align*}
     \partial_t \psi(y,t)={}&\sin(y-ct).
\end{align*}
From this we proceed to find their zeros. \\
\\
The identity $\partial_y \psi(y,t)=0$ holds, with $n\in \mathbb{Z}$, at
\begin{equation*}
\sin(\frac{ct}{2})=0 \quad\text{or}\quad \cos(\frac{ct}{2}-y)=0,
    \end{equation*}
that is
\begin{equation*}
t=\frac{2n\pi}{c} \quad\text{or}\quad y=\frac{ct}{2}-n\pi-\frac{\pi}{2},
    \end{equation*}
where we note that the first condition in time only. At these times the derivative vanishes independently of $y$ and hence we need to impose
\begin{align}
\label{mix:T_bound}
    T\leq \frac{2\pi}{c}.
\end{align}
This leaves the condition for $\partial_y \psi(y,t)=0$ as
\begin{align}\label{mix:py_zero}
    y=\frac{ct}{2}-n\pi-\frac{\pi}{2}.
\end{align}
$\partial_t \psi(y,t)=0$ holds at
\begin{equation}
\label{mix:pt_zero}
\sin(y-ct)=0,
\end{equation}
that is
\begin{equation}
 y=n\pi+ct.
\end{equation}
Now we introduce regions on which at least one derivative is bounded from below. Let $\delta \ll 1, \varepsilon \ll 1$ be small parameters. From \eqref{mix:py_zero} we define
\begin{align}
\label{mix:set_D}
  \mathcal{D}
  :=\left\{(y,t):\,\left|\cos(\frac{ct}{2}-y)\right|\ge\delta\right\}
\end{align}
and from \eqref{mix:pt_zero} we define
\begin{align}
\label{mix:set_E}
\mathcal{E}
  :=\left\{(y,t):\,\left|\cos(\frac{ct}{2}-y)\right|<\delta,\;|\sin(y-ct)|\ge\varepsilon\right\},
\end{align}
where the first condition in the definition excludes points already covered by $\mathcal{D}$.\\
Lastly, we have to consider a critical region $\mathcal C$ where no derivative is bounded from below, defined via
\begin{align}
\label{mix:set_C}
\mathcal C:=\left\{(y,t):\,
          \left|\cos(\frac{ct}{2}-y)\right|<\delta,\;
          |\sin(y-ct)|<\varepsilon\right\}.
\end{align}
Then the spcace-time strip $\T\cross [t_0,T]$ is partitioned by these sets as
\[
\T\cross [t_0,T]=\mathcal{D} \sqcup \mathcal{E} \sqcup \mathcal C.
\]

\paragraph{Step 2: Region-by-region estimates.}
We split the integral \eqref{mix:qty} into the regions identified above.
We now split \eqref{mix:qty} as:
\begin{align}
    \frac{1}{T}\int_{t_0}^T \int_{\mathbb T}
         e^{-ik \psi(y,t)}\phi(y)\,dy\, dt
             &= \underbrace{\frac{1}{T}\iint_{\mathcal{D}}   e^{-ik \psi(y,t)}\phi(y)\,dy \, dt}_{I_{\mathcal{D}}} \label{mix:term:D} \\
             &+ \underbrace{\frac{1}{T}\iint_{\mathcal{E}}  e^{-ik \psi(y,t)}\phi(y)\, dy \, dt}_{I_{\mathcal{E}}} \label{mix:term:E} \\
             &+ \underbrace{\frac{1}{T}\iint_{\mathcal{C}} e^{-ik \psi(y,t)}\phi(y)\,dy\, dt. }_{I_\mathcal{C}}\label{mix:term:C}
\end{align}
\emph{Estimate of $I_{\mathcal{D}}$}.\\
We write \eqref{mix:term:D} as
\begin{align*}
    I_{\mathcal{D}}=\frac{1}{T}\iint_{\mathcal{D}}   e^{-ik \psi(y,t)}\phi(y)\,dy \, dt={}&\frac{1}{T} \int_{t_0}^T\left( \sum_{j}\int_{y_j^-(t)}^{y_j^+(t)} e^{-ik \psi(y,t)}\phi(y)\, dy \right)\, dt,
\end{align*}
where the integration limits $y_j^\pm(t)$ can be directly computed from the definition of $\mathcal{D}$ (cf. \eqref{mix:set_D}) as
\begin{align*}
y_j^\pm(t) = \frac{ct}{2} + j\pi \pm \arccos(\delta).
\end{align*}
Therefore we can perform an integration by parts in $y$ on these intervals.
For this we use $\partial_y e^{-ik \psi} = -ik\,\partial_y \psi(y,t)\,e^{-ik \psi}$ to write
\[
   e^{-ik \psi(y,t)}
   = \frac{1}{-ik\,\partial_y \psi(y,t)}\,\partial_y\bigl(e^{-ik \psi(y,t)}\bigr).
\]
Hence, we get for the inner spatial integral of $I_{\mathcal{D}}$ that
\begin{align}
\label{mix:IG_Y_IBP}
   \sum_j &\int_{y_j^-}^{y_j^+} e^{-ik \psi(y,t)}\phi(y)\,dy\notag \\
   &= \sum_j \int_{y_j^-}^{y_j^+}
        \frac{\phi(y)}{-ik\,\partial_y \psi(y,t)}\,
        \partial_y\bigl(e^{-ik \psi(y,t)}\bigr)\,dy \notag \\
   &= \frac{1}{ik} \sum_j \left( \int_{y_j^-}^{y_j^+}
        e^{-ik \psi(y,t)}\,
        \partial_y\Bigl(\frac{\phi(y)}{\partial_y \psi(y,t)}\Bigr)\,dy
        +\left[\frac{\phi(y) e^{-ik \psi(y,t)}}{\partial_y \psi(y,t)}\right]_{y_j^-}^{y_j^+} \right).
\end{align}
Now we  use the fact that we have an explicit expression for $\psi(y,t)$ and are hence able to bound these terms directly. As the estimates are standard but lengthy, we refer to Appendix \ref{mix:Appendix-A} for a proof. From there we get via \eqref{mix:A:final}, that 
\begin{align}
\label{mix:Gy_est_pre}
    \abs{I_{\mathcal{D}}}\leq \frac{1}{T}\int_{t_0}^T\abs{\sum_j \int_{y_j^-}^{y_j^+} e^{-ik \psi(y,t)}\phi(y)\,dy\,  }dt \lesssim\frac{1}{T}\int_{t_0}^T \frac{1}{|k|t\delta}\,\|\widehat{\Theta}_0\|_{H^1_y}dt.
\end{align}
Note that the computations of Appendix \ref{mix:Appendix-A} force us to slightly restrict $\eqref{mix:T_bound}$ to $T\leq\frac{\pi}{c}$. If we further choose $t_0=1$ we get
\begin{align}
\label{mix:Gy_est}
     \abs{I_{\mathcal{D}}}\lesssim\frac{1}{T} \frac{\ln(T)}{|k|\delta}\,\|\widehat{\Theta}_0\|_{H^1_y}.
\end{align}
\emph{Estimate of $I_{\mathcal{E}}$}\\
We write \eqref{mix:term:E} as
\begin{align*}
  I_{\mathcal{E}}=\frac{1}{T}\iint_{\mathcal{E}}  e^{-ik \psi(y,t)}\phi(y)\, dy \, dt
  &=\frac{1}{T}\int_{\T} \left( \sum_j \int_{t_{m,n}^-(y)}^{t_{m,n}^+(y)} e^{-ik \psi(y,t)}\phi(y)\, dt \right)\, dy.
\end{align*}
For a fixed connected component of $\mathcal{E}$ indexed by $m,n \in \mathbb{Z}$, the integration interval $[t_{m,n}^-(y), t_{m,n}^+(y)]$ is explicitly defined by the two inequalities defining $\mathcal{E}$. We note that the endpoints can be defined via indices $m,n\in\mathbb Z$ such that:
\begin{align}
  t_{m,n}^-(y)
  &= \max\left\{
      \frac{2}{c}\bigl(y+\arccos\delta+m\pi\bigr),\;
      \frac{1}{c}\bigl(y-(n+1)\pi+\arcsin\varepsilon\bigr)
    \right\},\\
  t_{m,n}^+(y)
  &= \min\left\{
      \frac{2}{c}\bigl(y+\pi-\arccos\delta+m\pi\bigr),\;
      \frac{1}{c}\bigl(y-n\pi-\arcsin\varepsilon\bigr)
    \right\}.
\end{align}

On these intervals we have $|\partial_t \psi(y,t)|\geq \varepsilon$ and hence we are able to integrate by parts in time. Further, note that the second derivative is bounded from above as
\[
\partial_t^2\psi(y,t)=-c\cos(y-ct) \ \implies \abs{\partial_t^2\psi(y,t)}\leq c.
\]
Hence we get for the inner time integral of $I_{\mathcal{E}}$ that
\begin{align}
\label{mix:Gy_IBP}
  \sum_j &\int_{t_j^-}^{t_j^+} e^{-ik \psi(y,t)}\phi(y)\, dt\notag \\
  ={}& \frac{\phi(y)}{ik} \sum_j \int_{t_j^-}^{t_j^+}
        \frac{1}{\partial_t\psi(y,t)}\,
        \partial_t\bigl(e^{-ik\psi(y,t)}\bigr)\,dt\notag \\
  = &\frac{\phi(y)}{ik} \sum_j \left(
     \left[\frac{e^{-ik\psi(y,t)}}{\partial_t\psi(y,t)}\right]_{t_j^-}^{t_j^+}
     - \int_{t_j^-}^{t_j^+}
        \frac{\partial_t^2\psi(y,t)}{(\partial_t\psi(y,t))^2}
        e^{-ik\psi(y,t)}\,dt \right).
\end{align}
In order to estimate the integral term, note that the intervals $[t_{m,n}^-(y), t_{m,n}^+(y)]$ are guaranteed to be contained in a region where $|\cos(\frac{ct}{2}-y)| < \delta$ (cf. \eqref{mix:set_E}). The endpoints of such a region are given by
\[
   t^\pm(y) = \frac{2}{c}\left(y+\frac{\pi}{2}\right) \pm \frac{2}{c}\arcsin(\delta).
\]
Consequently, for all $y\in \mathcal{E}$, we get
\[
 |t_{m,n}^+(y) - t_{m,n}^-(y)|
   \le |t^+(y) - t^-(y)|
   = \frac{4}{c}\arcsin(\delta)
   \lesssim \frac{\delta}{c},
\]
where we used that $\delta \ll 1$ and therefore approximated $\arcsin(\delta)\sim \delta.$\\
Then, by recalling that $|\partial_t \psi(y,t)|\geq \varepsilon$ and
 $\abs{\partial_t^2\psi(y,t)}\leq c$, we  find
  \begin{align*}
  \sum_j \Bigl|\int_{t_j^-}^{t_j^+}
        \frac{\partial_t^2\psi(y,t)}{(\partial_t\psi(y,t))^2}
        e^{-ik\psi(y,t)}\,dt\Bigr|
  &\le \frac{c}{\varepsilon^2}\,\sum_j \Bigl|\int_{t_j^-}^{t_j^+} 1\, dt\Bigr|
   \lesssim \frac{c}{\varepsilon^2} \frac{\delta}{c}=\frac{\delta}{\varepsilon^2}.
\end{align*}
Thus \eqref{mix:Gy_IBP} can be estimated as
\[
  \Bigl|\sum_j \int_{t_j^-}^{t_j^+} e^{-ik\psi(y,t)}\phi(y)\,dt\Bigr|
 \lesssim
  \frac{|\phi(y)|}{|k|}\left(\frac{1}{\varepsilon}
                    + \frac{\delta}{\varepsilon^2}\right).
\]
We then can get an estimate for $I_{\mathcal{E}}$ as
\begin{align*}
  |I_{\mathcal{E}}|
  &= \left|\frac{1}{T}\int_{\T} \left( \sum_j \int_{t_j^-}^{t_j^+} e^{-ik \psi(y,t)}\phi(y)\, dt \right)\, dy\right| \\
  &\lesssim \frac{1}{|k|\,T}\left(\frac{1}{\varepsilon}
                    + \frac{\delta}{\varepsilon^2}\right)
        \int_{\T}|\phi(y)|\,dy \\
  &\lesssim \frac{1}{|k|\,T}\left(\frac{1}{\varepsilon}
                    + \frac{\delta}{\varepsilon^2}\right)
        \|\phi\|_2
   \lesssim
   \frac{1}{|k|\,T}\left(\frac{1}{\varepsilon}
                    + \frac{\delta}{\varepsilon^2}\right)
   \|\widehat{\Theta}_0\|_{H^1_y},
\end{align*}
where we used $\|\phi\|_{2}\leq \|\widehat{\Theta}_0\|_\infty\|\eta\|_{2}\lesssim\|\widehat{\Theta}_0\|_{H^1_y}\|\eta\|_{H^1_y}$ and
$\|\eta\|_{H^1}=1$.
We conclude that
\begin{align}
\label{mix:Gt_est}
    |I_{\mathcal{E}}|\lesssim
   \frac{1}{|k|\,T}\left(\frac{1}{\varepsilon}
                    + \frac{\delta}{\varepsilon^2}\right)
   \|\widehat{\Theta}_0\|_{H^1_y}.
\end{align}
\\
\emph{Estimate of $I_{\mathcal C}$}.\\
On $\mathcal{C}$ we experience double-degeneracies of the derivatives and hence neither integration by parts is available. Therefore, we simply bound the double integral by the regions measure. We get
\begin{align*}
  I_\mathcal{C}= \frac{1}{T}\iint_{\mathcal{C}} e^{-ik \psi(y,t)}\phi\,dy\, dt\leq \frac{1}{T} \|\phi\|_{L^\infty_y} \cdot \operatorname{meas}(\mathcal{C}).
\end{align*}
To compute the measure, we use the change of variables $u = \frac{ct}{2} - y$ and $v = y - ct$. The Jacobian determinant is calculated as
\[
    \det \frac{\partial(u,v)}{\partial(y,t)} = \det \begin{pmatrix} -1 & c/2 \\ 1 & -c \end{pmatrix} = \frac{c}{2},
\]
which yields the volume element $dy\,dt = \frac{2}{c}\,du\,dv$. 

The region $\mathcal{C}$ is then defined by $|\cos u| < \delta$ and $|\sin v| < \varepsilon$. These inequalities correspond to the intervals
\[
    u \in (\arccos \delta, \arccos(-\delta)) \quad \text{and} \quad v \in (\arcsin(-\varepsilon), \arcsin \varepsilon).
\]
Using these explicit bounds, we compute the measure as:
\begin{align*}
    \operatorname{meas}(\mathcal{C}) 
    &= \frac{2}{c} \left( \int_{\arccos \delta}^{\arccos(-\delta)} du \right) \left( \int_{\arcsin(-\varepsilon)}^{\arcsin \varepsilon} dv \right) \\
    &= \frac{2}{c} \Bigl[ \arccos(-\delta) - \arccos(\delta) \Bigr] \Bigl[ \arcsin(\varepsilon) - \arcsin(-\varepsilon) \Bigr].
\end{align*}
For small $\delta, \varepsilon \ll 1$, we use the linear approximations $\arccos(\pm \delta) \approx \frac{\pi}{2} \mp \delta$ and $\arcsin(\pm \varepsilon) \approx \pm \varepsilon$ to find
\[
   \operatorname{meas}(\mathcal{C}) \approx \frac{2}{c} (2\delta)(2\varepsilon) \lesssim\frac{\delta\varepsilon}{c}.
\]
Using again the embedding $H^1(\T)\hookrightarrow L^\infty(\T)$ we then find that 
\begin{align}\label{mix:C_est}
     |I_\mathcal{C}|\leq \frac{1}{T} \cdot \frac{\delta \varepsilon}{c} \|\widehat{\Theta}_0\|_{H^1_y}.
\end{align}
\paragraph{Step 3: Optimisation.}
From \eqref{mix:Gy_est}, \eqref{mix:Gt_est} and \eqref{mix:C_est}, we can find an estimate of \eqref{mix:qty} as
\begin{multline}
\label{mix:pre_meas}
     \abs{\frac{1}{T}\int_{t_0}^T \int_{\mathbb T}
         e^{-ik \psi(y,t)}\phi(y)\,dy\, dt}\\
         \lesssim \frac{1}{T}\qty[\frac{\ln(T)}{|k|\delta}+\frac{1}{|k|}\left(\frac{1}{\varepsilon}
                    + \frac{\delta}{\varepsilon^2}\right)
   +
   \frac{\delta\varepsilon}{c}]\,\|\widehat{\Theta}_0\|_{H^1_y}.
\end{multline}
We now choose $\delta \sim \epsilon$ to balance all contributions. Then \eqref{mix:pre_meas} simplifies to
\begin{align*}
    \label{mix:post_meas}
     \abs{\frac{1}{T}\int_{t_0}^T \int_{\mathbb T}
         e^{-ik \psi(y,t)}\phi(y)\,dy\, dt}
    \lesssim \frac{1}{T}\qty[\frac{\ln(T)}{|k|\delta}
   +\frac{\delta^2}{c}]\,\|\widehat{\Theta}_0\|_{H^1_y}.
\end{align*}
Taking the supremum over all $\eta\in H^1_y(\mathbb{T})$ with $\|\eta\|_{H^1_y}=1$ now yields
\begin{equation}\label{mix:eq:A-bound}
  \biggl\|\frac{1}{T}\int_1^T \widehat{\Theta}(k,\cdot,t)\,dt\biggr\|_{H^{-1}_y}
  \lesssim
  \frac{1}{T}\qty[\frac{\ln(T)}{|k|\delta}
   +\frac{\delta^2}{c}]\|\widehat{\Theta}_0(k,\cdot)\|_{H^1_y}.
\end{equation}
We now minimise the right-hand side of \eqref{mix:eq:A-bound} with respect to
$\delta>0$. Setting
\[
  F(\delta)
  := \frac{\ln T}{|k|\,T\,\delta} + \frac{\delta^2}{c\,T},
\]
we have
\[
  F'(\delta)
  = -\frac{\ln T}{|k|\,T}\,\frac{1}{\delta^2}
    + \frac{2\delta}{c\,T},
\]
so that $F'(\delta)=0$ gives
\[
  \delta^3
  = \frac{c\,\ln T}{2|k|},
  \qquad
  \delta \sim \left(\frac{c\,\ln T}{|k|}\right)^{1/3}.
\]
Using this choice of $\delta$ into \eqref{mix:eq:A-bound} yields the final expression
\begin{equation}\label{mix:eq:A-final}
  \biggl\|\frac{1}{T}\int_1^T \widehat{\Theta}(k,\cdot,t)\,dt\biggr\|_{H^{-1}_y}
  \;\lesssim\;
  \frac{1}{T}\left(\frac{(\ln T)^2}{c\,|k|^2}\right)^{1/3}\|\widehat{\Theta}_0(k,\cdot)\|_{H^1_y},
  \qquad T> 1.
\end{equation}
This concludes the proof.

\section{Intermediate translation speeds}
\label{sec:hyp}

In this section we aim to prove Theorem \ref{hyp:thm:L2_decay}.
We take the Fourier transform of \eqref{adv} in the \(x\)-variable decouples the Fourier modes \(k\in\mathbb Z\). Consequently the Fourier coefficients \(\widehat{\Theta}(k,y,t)\) evolve independently according to the mode-by-mode equation
\begin{equation*}
  \partial_t \widehat{\Theta}(k,y,t) + i \alpha k \sin(y-ct) \widehat{\Theta}(k,y,t) + \nu k^2 \widehat{\Theta}(k,y,t) - \nu \partial_y^2 \widehat{\Theta}(k,y,t) = 0.
\end{equation*}
The mean-free condition \eqref{mean_free} implies that the mode \(k=0\) vanishes.
For the purpose of the proof we temporarily pass to the moving frame 
\begin{equation}
\label{hyp:cov}
  (y,t)\mapsto (y+ct,\;\tau),\qquad \tau:=\alpha \abs{k}\,t,
\end{equation}
 where the shear becomes stationary. 
All estimates are derived in this coordinate system and then translated 
back to the original variables when stating the result.

We define
\begin{equation}
\label{hyp:remove_exp}
  \theta(k,y,\tau)\;:=\;e^{\nu k^2 t}\,\widehat{\Theta}\bigl(k,\,y+ct,\,t\bigr),
\end{equation}
where, in a slight abuse of notation, we continue to write $y$ for the spatial variable in the moving frame.
Further, as outlined in the Introduction, we assume that
\begin{equation}
\label{hyp:c_original}
c =  c_0\,\nu^{\ell}.
\end{equation}
Now we introduce the parameter rescaling
\[
  \varsigma=\frac{c}{\alpha k},\qquad \mu=\frac{\nu}{\alpha k},
\]
then $\theta(k,y,\tau)$ solves
\begin{equation}\label{hyp:eq}
  \partial_\tau\theta(k,y,\tau) - \varsigma\partial_y\theta(k,y,\tau) + i\sin(y)\theta(k,y,\tau) = \mu\partial_y^2\theta(k,y,\tau).
\end{equation}
with
\begin{equation}
\label{hyp:c_0}
  \varsigma = \varsigma_{k}\,\mu^{\ell},
  \qquad
  \varsigma_{k}:= c_0(\alpha|k|)^{\ell-1}.
\end{equation}
We note that the rescaling preserves the power law as $\varsigma \sim \mu^\ell$. We now can state the following Proposition.
\begin{prop}\label{hyp:thm:L2_decay}
Let $\ell\in\big(\frac13,\frac34\big)$ and $\beta_0\in(0,1)$. There exists a threshold $\mu_0\in(0,1)$ such that for all $k\in\mathbb Z\setminus\{0\}$ and $0<\mu\le\mu_0$, assuming the parameter scaling
\[
  \varsigma = c_0(\alpha|k|)^{\ell-1}\mu^\ell,
\]
the solution $\theta(k,y,\tau)$ to \eqref{hyp:eq} with initial datum $\theta_0\in L^2(\T)$ satisfies
\begin{equation}\label{hyp:L2_decay}
  \|\theta(k,\cdot,\tau)\|_2^2
  \;\le\;
  C_{\mathrm{ed}}\Bigl(1+\beta_0^{1/2}\mu^{-\frac{1+2\ell}{5}}\Bigr)
  \exp\!\Bigl(-\frac{\beta_0^{5/4}}{12C_s}\,\varsigma_k\,\mu^{\frac{1+2\ell}{5}}\,\tau\Bigr)\,
  \|\theta_0(k,\cdot)\|_2^2
\end{equation}
for all $\tau\ge0$. Here, $C_{\mathrm{ed}}>0$ and $C_s\ge 1$ are constants independent of $\mu, k, \tau$, and we denoted $\varsigma_k \coloneqq c_0(\alpha|k|)^{\ell-1}$.
\end{prop}
By undoing the change of variables and parameter rescalings in \eqref{hyp:cov}--\eqref{hyp:c_0} Theorem \ref{hyp:thm:L2_phys} follows directly from this Proposition.

Before we can move to the proof of this theorem, we need to introduce the energy functional 
\begin{align}\label{hyp:functional}
	\Phi \coloneqq \frac{1}{2}\bigl[E_0(\tau) + \alpha_0  E_1(\tau) + 2\beta_0  E_3(\tau) + \gamma_0   E_4(\tau)- 2\beta_1   E_6(\tau)+ \gamma_1   E_7(\tau)\bigr],
\end{align}
where 
\begin{equation}\label{hyp:Es}
\begin{alignedat}{2}
  E_0(\tau) &= \|\theta(k,\cdot,\tau)\|_2^{2},              &\qquad
  E_1(\tau) &= \|\partial_y\theta(k,\cdot,\tau)\|_2^{2},    \\[5pt]
  E_2(\tau) &= \|\partial_y^2\theta(k,\cdot,\tau)\|_2^{2},  &\qquad
  E_3(\tau) &= \Re\langle i\partial_y v\,\theta(k,\cdot,\tau),\, \partial_y\theta(k,\cdot,\tau)\rangle, \\[5pt]
  E_4(\tau) &= \|\partial_y v\,\theta(k,\cdot,\tau)\|_2^{2}, &\qquad
  E_6(\tau) &= \Re\langle v\,\theta(k,\cdot,\tau),\, \partial_y v\,\theta(k,\cdot,\tau)\rangle, \\[5pt]
  E_7(\tau) &= \|v\,\theta(k,\cdot,\tau)\|_2^{2}.
\end{alignedat}
\end{equation}

and we defined
\begin{equation}
    \label{v:def}
     v(y)=\sin(y).
\end{equation}
We also record a structural consequence of our choice of parameters, which will be used repeatedly in the remainder of this section. In view of the parameter choice \eqref{hyp:param_choice}, the coefficients $\alpha_0,\beta_0,\gamma_0,\beta_1,\gamma_1$ satisfy the inequalities
\begin{align}
	\label{hyp:param_ineq}
    \frac{\beta_0^2}{\alpha_0}\leq \frac{\gamma_0}{16},
    \qquad
    \beta_1^2\leq \gamma_0 \gamma_1.
\end{align}
These relations guaranteed coercivity of the energy functional $\Phi$, in the sense that
\begin{equation}\label{hyp:coercivity}
\begin{multlined}
\frac{1}{8}\Bigl(4E_0+3\alpha_0E_1+2\gamma_0E_4+3\gamma_1E_7\Bigr)
\;\leq\; \Phi \\
\leq\;
\frac{1}{8}\Bigl(4E_0+5\alpha_0E_1+6\gamma_0E_4+5\gamma_1E_7\Bigr).
\end{multlined}
\end{equation}

Indeed, by Young's inequality and \eqref{hyp:param_ineq},
\begin{align*}
	2\beta_0 |E_3|
	\leq \frac{\alpha_0}{4}E_1+\frac{4\beta_0^2}{\alpha_0}E_4
	\leq \frac{\alpha_0}{4}E_1+\frac{\gamma_0}{4}E_4.
\end{align*}
Similarly, using \eqref{hyp:param_choice} and another application of Young's inequality leads to
\begin{align*}
	2\beta_1 |E_6|
	= 2 \frac{\beta_0^{3/4}}{\mu^{p/2}}\frac{\beta_0^{1/4}}{\mu^{q/2}}\,|E_6|
	\leq 4\frac{\beta_0^{3/2}}{\mu^{p}} E_4+\frac{\beta_0^{1/2}}{4\mu^{q}} E_7
	\leq \frac{\gamma_0}{4} E_4+\frac{\gamma_1}{4} E_7,
\end{align*}
where in the last step we used again \eqref{hyp:param_ineq}. From this \eqref{hyp:coercivity} follows immediately.
\\
\\
Now with this established, we can introduce the following Proposition, capturing the decay of the energy functional. This is fundamental for the proof of Theorem \ref{hyp:thm:L2_decay}.
\begin{prop}\label{hyp:prop:hypo_decay}
Let $\mu\ll 1$, $\Phi$ as defined in \eqref{hyp:functional},
\[
c=\varsigma_k\,\mu^{\ell}
\]
for some $\ell\in\big(\frac13,\frac34\big)$ and $\varsigma_k>0$ defined in \eqref{hyp:c_0}.
Further, choose the coefficients in \eqref{hyp:functional} as
\begin{align}\label{hyp:param_choice}
\alpha_0=\beta_0^{1/2} \mu^{p},\quad
\gamma_0=16\frac{\beta_0^{3/2}}{\mu^{p}},\quad
\beta_1=\frac{\beta_0}{\mu^{\frac{p+q}{2}}},\quad
\gamma_1=\frac{\beta_0^{1/2}}{\mu^{q}},
\end{align}
where \(\beta_0\in(0,1)\) is a fixed constant and 
\begin{equation}
\label{hyp:p_q_choice}
    p=\frac{1+2\ell}{5},\qquad q=\frac{3-4\ell}{5}.
\end{equation}
Then there exist a constant $\mu_0\in(0,1)$ such that for all $0<\mu\le\mu_0$,
\[
\dv{}{\tau}\Phi(\tau)+  \frac{\beta_0^{5/4}}{12C_{s}}\varsigma_k\mu^{\frac{1+2\ell}{5}}\Phi(\tau)\leq 0,
\]
where $C_{s} \ge 1$ is a constant independent of $\mu, \varsigma$.
\end{prop}

\subsection{Proof of Theorem \ref{hyp:thm:L2_decay}}
The proof proceeds in three steps.
\paragraph{Step 1: Gradient Control}
Recall the $L^2$-energy balance from Lemma~\ref{hyp:L:energy}:
\[
  \frac12\frac{d}{dt}E_0(\tau) = -\mu E_1(\tau),
  \qquad E_0(\tau) = \|\theta(k,\cdot,\tau)\|_2^2, \quad E_1(\tau) = \|\partial_y\theta(k,\cdot,\tau)\|_2^2.
\]
We now define the timescale:
\begin{equation}\label{hyp:def:Tburn}
  T_\mu := \mu^{-(1-2p)}.
\end{equation}
Note that from Proposition \ref{hyp:prop:hypo_decay} we have the admissible range $\ell \in (1/3, 3/4)$ and by \eqref{hyp:p_q_choice} this implies $p \in (1/3, 1/2)$, which ensures $1-2p>0$.
Integrating the above identity over $(0, T_\mu)$ we find
\[
  \int_{0}^{T_\mu} E_1(s)\,ds = \frac{E_0(0) - E_0(T_\mu)}{2\mu} \le \frac{E_0(0)}{2\mu}.
\]
By the Mean Value Theorem, there exists a time $\tau_0 \in (0, T_\mu)$ such that:
\begin{equation}\label{hyp:E1_at_t0}
  E_1(\tau_0) \le \frac{1}{T_\mu}\int_{0}^{T_\mu} E_1(s)\,ds \le \frac{E_0(0)}{2\mu T_\mu} = \frac12 \mu^{-2p} E_0(0).
\end{equation}

\paragraph{Step 2: Bounding the functional at $\tau_0$.}
Using the upper coercivity bound \eqref{hyp:coercivity} at $\tau_0$ gives
\[
  \Phi(\tau_0) \le \frac{1}{8}\Bigl(4E_0(\tau_0) + 5\alpha_0E_1(\tau_0) + 6\gamma_0E_4(\tau_0) + 5\gamma_1E_7(\tau_0)\Bigr).
\]
Now for $v(y) = \sin y$, we have $\|\partial_y v\|_{L^\infty_y} = \|v\|_{L^\infty_y} = 1$, implying $E_4(\tau) \le E_0(\tau)$ and $E_7(\tau) \le E_0(\tau)$.
Further, since $E_0(\tau)$ is non-increasing we have $E_0(\tau_0) \le E_0(0)$.
These observations together with \eqref{hyp:E1_at_t0} and the parameter choices \eqref{hyp:param_choice} yield
\[
  \alpha_0 E_1(\tau_0) \le (\beta_0^{1/2}\mu^p) \cdot (\frac{1}{2}\mu^{-2p} E_0(0)) = \frac{1}{2}\beta_0^{1/2}\mu^{-p} E_0(0),
\]
\[
  \gamma_0 E_4(\tau_0) \le 16\beta_0^{3/2}\mu^{-p} E_0(0), \qquad
  \gamma_1 E_7(\tau_0) \le \beta_0^{1/2}\mu^{-q} E_0(0).
\]
By \eqref{hyp:param_choice} in Proposition \ref{hyp:prop:hypo_decay} we have $p \ge q$ for $\ell \ge 1/3$ and we therefore can bound $\mu^{-q} \le \mu^{-p}$ for $\mu\in(0,1]$.
Moreover, $\beta_0\in(0,1]$ is a fixed constant. Therefore there exists $C_1>0$ such that
\begin{equation}\label{hyp:Phi_t0_bound}
  \Phi(\tau_0) \le C_1 \Bigl(1 + \beta_0^{1/2}\mu^{-p}\Bigr) E_0(0).
\end{equation}

\paragraph{Step 3: Global in time estimate.}
Define
\begin{equation}\label{hyp:def:lambda_mu}
  \lambda_\mu := \frac{1}{12C_s}\,\beta_0^{5/4}\,\varsigma_k\,\mu^{p}.
\end{equation}
By Proposition~\ref{hyp:prop:hypo_decay}, for all $\tau\ge \tau_0$,
\begin{equation}\label{hyp:eq:Phi_decay}
  \Phi(\tau)\le e^{-\lambda_\mu(\tau-\tau_0)}\,\Phi(\tau_0).
\end{equation}
Using the lower coercivity bound in \eqref{hyp:coercivity}, namely $E_0(\tau)\le 2\Phi(\tau)$, it follows that
\begin{equation}\label{hyp:eq:E0_from_Phi}
  E_0(\tau)\le 2e^{-\lambda_\mu(\tau-\tau_0)}\Phi(\tau_0)
  = 2e^{\lambda_\mu\tau_0}e^{-\lambda_\mu\tau}\Phi(\tau_0),
  \qquad \tau\ge\tau_0.
\end{equation}

A uniform bound on $e^{\lambda_\mu\tau_0}$ is obtained as follows. Since $\tau_0\le T_\mu$,
\[
  \lambda_\mu\tau_0 \le \lambda_\mu T_\mu
  = \frac{1}{12C_s}\beta_0^{5/4}\varsigma_k\,\mu^{p}\mu^{-(1-2p)}
  = \frac{1}{12C_s}\beta_0^{5/4}\varsigma_k\,\mu^{3p-1}.
\]
Because $p\in(1/3,1/2)$, one has $3p-1>0$, and therefore $\mu^{3p-1}\le 1$ for $\mu\in(0,1]$.
Using $\varsigma_k\le \varsigma_1=c_0\alpha^{\ell-1}$ from \eqref{hyp:c_0}, define
\begin{equation}\label{hyp:def:C2}
  C_2 := \exp\!\Bigl(\frac{1}{12C_s}\beta_0^{5/4}c_0\alpha^{\ell-1}\Bigr).
\end{equation}
Then $e^{\lambda_\mu\tau_0}\le C_2$.

Substituting $e^{\lambda_\mu\tau_0}\le C_2$ into \eqref{hyp:eq:E0_from_Phi} and using \eqref{hyp:Phi_t0_bound} yields, for all $\tau\ge\tau_0$,
\[
  E_0(\tau)
  \le 2C_2 e^{-\lambda_\mu\tau}\Phi(\tau_0)
  \le 2C_1C_2\Bigl(1+\beta_0^{1/2}\mu^{-p}\Bigr)e^{-\lambda_\mu\tau}E_0(0).
\]

For $\tau\in[0,\tau_0]$, monotonicity of $E_0$ implies $E_0(\tau)\le E_0(0)$.
Since $\tau\le\tau_0$, one has $e^{\lambda_\mu(\tau_0-\tau)}\ge 1$ and hence
\[
  E_0(\tau)\le E_0(0)
  \le e^{\lambda_\mu(\tau_0-\tau)}E_0(0)
  = e^{\lambda_\mu\tau_0}e^{-\lambda_\mu\tau}E_0(0)
  \le C_2 e^{-\lambda_\mu\tau}E_0(0).
\]

Now define
\begin{equation}\label{hyp:def:Ced}
  C_{\mathrm{ed}} := 2C_2\max\{C_1,1\},
\end{equation}
then, for all $\tau\ge 0$,
\[
  E_0(\tau)
  \le C_{\mathrm{ed}}\Bigl(1+\beta_0^{1/2}\mu^{-p}\Bigr)e^{-\lambda_\mu\tau}E_0(0).
\]
Recall that $E_0(\tau)=\|\theta(k,\cdot,\tau)\|_2^2$ and $E_0(0)=\|\theta_0\|_2^2$. Then by invoking the definitions \eqref{hyp:p_q_choice} and \eqref{hyp:def:lambda_mu} we obtain \eqref{hyp:L2_decay}.

\subsubsection*{Proof of Proposition \ref{hyp:prop:hypo_decay} - Hypocoercivity}
We organise the proof in three steps. First, we differentiate the functional \eqref{hyp:functional} in time and derive a system of energy identities; the resulting error terms are then estimated individually. Second, we apply the spectral estimate of Lemma~\ref{hyp:L:spectral_gap} to recover a form compatible with the coercivity inequality \eqref{hyp:coercivity}. Finally, we choose the auxiliary parameters to satisfy the constraints introduced along the argument and to optimize the resulting bound.\\

\paragraph{Step 1: Time-derivative of $\Phi$.}
We start by recalling the following energy balances provided by Lemma \ref{hyp:L:energy} in the Appendix.
\begin{enumerate}
		\item $\displaystyle\frac{1}{2}\dv{}{\tau}E_0 = -\mu E_1,$
		\item $\displaystyle\frac{1}{2}\dv{}{\tau}E_1 = -\mu E_2 - E_3,$
		\item $\displaystyle\dv{}{\tau}E_3 = -E_4 - 2\mu\Re\inner{i\partial_y v\partial_y \theta}{\partial_y^2 \theta} - \mu\Re\inner{i\partial_y^2 v\theta }{\partial_y^2 \theta}-\varsigma\Re \inner{i\partial_{y}^2 v \theta}{\partial_{y} \theta},$
		\item $\displaystyle\frac{1}{2}\dv{}{\tau}E_4 = -\mu \norm{\partial_y v\partial_y \theta}_2^2 - 2\mu \Re \inner{\partial_y v\partial_y^2 v \theta}{\partial_y \theta}+\varsigma E_6,$
        \item $\displaystyle\dv{}{\tau} E_6= -\varsigma E_4+\varsigma E_7 -4 \mu E_6-2 \mu \inner{v \partial_y \theta}{\partial_y v \partial_y \theta},  $
        \item $\displaystyle \frac{1}{2}\dv{}{\tau}E_7=-\varsigma E_6-\mu \norm{v \partial_y \theta}_2^{2}+\mu E_4 -\mu E_7$.
	\end{enumerate}
We now take the energy functional $\Phi$ defined in \eqref{hyp:functional} and compute it's time derivative.
\begin{align*}
     \dv{}{\tau}\Phi&+\mu E_1+\mu \alpha_0 E_2 + \beta_0  E_4\\
        ={}& -\alpha_0 E_3 -2\mu \beta_0 \Re\inner{i\partial_y v \partial_y\theta}{\partial_{y}^2 \theta}-\mu \beta_0\Re\inner{i \partial_{y}^2v \theta}{\partial_{y}^2 \theta}\\
        &-\mu\gamma_0\norm{\partial_y v\partial_{y} \theta}_2^{2}-2\mu\gamma_0 \Re \inner{\partial_y v\partial_{y}^2v \theta}{\partial_{y} \theta}\\
        &-\varsigma\beta_0\Re \inner{i\partial_{y}^2 v \theta}{\partial_{y} \theta}+\varsigma \gamma_0E_6\\
        &+\beta_1 \varsigma E_4-\beta_1 
        \varsigma E_7 + \beta_1 4 \mu E_6+\beta_1 2 \mu \inner{v \partial_y \theta}{\partial_y v \partial_y \theta}\\
        &-\varsigma\gamma_1 E_6-\gamma_1\mu \norm{v \partial_y \theta}_2^{2}+\gamma_1\mu E_4 -\gamma_1\mu E_7.
\end{align*}
We can simplify this by collecting terms as
\begin{equation}
\label{hyp:Phi_original}
\begin{aligned}
     \dv{}{\tau}\Phi&+\mu E_1+\mu \alpha_0 E_2 + (\beta_0-B_4)  E_4+B_7 E_7\\
        ={}& -\alpha_0 E_3 -2\mu \beta_0 \Re\inner{i\partial_y v \partial_y\theta}{\partial_{y}^2 \theta}-\mu \beta_0\Re\inner{i \partial_{y}^2v \theta}{\partial_{y}^2 \theta}\\
        &-\mu\gamma_0\norm{\partial_y v\partial_{y} \theta}_2^{2}-2\mu\gamma_0 \Re \inner{\partial_y v\partial_{y}^2v \theta}{\partial_{y} \theta}\\
        &-\varsigma\beta_0\Re \inner{i\partial_{y}^2 v \theta}{\partial_{y} \theta}+B_6E_6\\
        &+\beta_1 2 \mu \inner{v \partial_y \theta}{\partial_y v \partial_y \theta}
        -\gamma_1\mu \norm{v \partial_y \theta}_2^{2},
\end{aligned}
\end{equation}
where we defined
\begin{align*}
    B_4\coloneqq\beta_1 \varsigma +\gamma_1 \mu,\quad
  B_7\coloneqq\beta_1 \varsigma + \gamma_1 \mu, \quad
  B_6\coloneqq \varsigma \gamma_0-\varsigma\gamma_1+4\beta_1  \mu.
\end{align*}
\emph{Fixing the parameter regime.}\\
We proceed by simplifying $B_4$, $B_6$, and $B_7$ for the range $0<\ell <1$.
Recall that by \eqref{hyp:param_choice} and \eqref{hyp:c_0} we can write
\begin{subequations}
\begin{align}
  B_4 &= \beta_1 \varsigma +\gamma_1 \mu
       = \beta_0\varsigma_k\mu^{\frac{2\ell-p-q}{2}}
         +\beta_0^{1/2}\mu^{1-q}, \label{hyp:B4} \\[0.3em]
  B_7 &= \beta_1 \varsigma + \gamma_1 \mu
       = \beta_0 \varsigma_k\mu^{\frac{2\ell-p-q}{2}}
         +\beta_0^{1/2}\mu^{1-q}, \label{hyp:B7} \\[0.3em]
  B_6 &= \varsigma \gamma_0-\varsigma\gamma_1+4\beta_1  \mu
       = \beta_0^{3/2}\varsigma_k\mu^{\ell-p}
         -\beta_0^{1/2}\varsigma_k\mu^{\ell-q}
         +4\beta_0\mu^{\frac{2-p-q}{2}} \label{hyp:B6}
\end{align}
\end{subequations}

As $\mu \ll 1$, we need to keep all exponents of $\mu$ non-negative to avoid blow-up. For this we recall that $0< p,q <1$ and hence require
\begin{align}
\label{hyp:pq_y_fix}
    \ell-p\geq 0 \iff p\leq \ell,\quad \ell-q\geq 0\iff q\leq \ell.
\end{align}
Further, since the sign in front of $B_4$ is negative in \eqref{hyp:Phi_original} we need to bound it.
Since $\beta_0$ is a constant and $\varsigma_k$ is bound by a constant as $\varsigma_k\leq \varsigma_1$ (cf. \eqref{hyp:c_0}),
we can choose a sufficiently small $\mu_0\in(0,1)$ so that for all $0<\mu\le \mu_0$,
\[
\mu^{\frac{2\ell-p-q}{2}}\le \frac{1}{8 \varsigma_1}\,
\]
and
\[
\mu^{1-q}\le \frac{\beta_0^{1/2}}{8}.
\]
Hence
\begin{equation}
\label{hyp:B4_bound}
B_4=\beta_0 \varsigma_k\mu^{\frac{2\ell-p-q}{2}}+\beta_0^{1/2}\mu^{1-q}
\le \frac14\,\beta_0.
\end{equation}
We then can apply \eqref{hyp:B4_bound} to \eqref{hyp:Phi_original} to get
\begin{equation}
\label{hyp:Phi_cleaned}
\begin{aligned}
     \dv{}{\tau}\Phi&+\mu E_1+\mu \alpha_0 E_2 + \frac{3}{4}\beta_0 E_4+ B_7 E_7\\
        \leq{}& -\alpha_0 E_3 -2\mu \beta_0 \Re\inner{i\partial_y v \partial_y\theta}{\partial_{y}^2 \theta}-\mu \beta_0\Re\inner{i \partial_{y}^2v \theta}{\partial_{y}^2 \theta}\\
        &-\mu\gamma_0\norm{\partial_y v\partial_{y} \theta}_2^{2}-2\mu\gamma_0 \Re \inner{\partial_y v\partial_{y}^2v \theta}{\partial_{y} \theta}\\
        &-\varsigma\beta_0\Re \inner{i\partial_{y}^2 v \theta}{\partial_{y} \theta}-B_6 E_6\\
        &+\beta_1 2 \mu \inner{v \partial_y \theta}{\partial_y v \partial_y \theta}-\gamma_1\mu \norm{v \partial_y \theta}_2^{2}.
\end{aligned}
\end{equation}
\paragraph{Step 2: Control of errors.}
From this simplified form, we next turn to the control of the error terms. We define as error terms all contributions that do not appear in the coercivity estimate \eqref{hyp:coercivity}, i.e.\ all terms other than $E_0, E_1, E_4,$ and $E_7$. We now estimate each such error term by the corresponding coercive norms.
\\
\\
\emph{Control standard error terms.}\\
We first note that the contributions in the first two lines on the right-hand side of \eqref{hyp:Phi_cleaned} are classical, in the sense that they appear in the established literature on autonomous systems (cf. for example \cite{Bedrossian2017, cotizelatigallay}). Therefore they can be controlled by standard Young-type estimates.
 Briefly, these are
\begin{align*}
		\lvert \alpha_0 E_3 \rvert
		&\leq \frac{1}{2}\frac{\alpha_0^2}{\beta_0}E_1+\frac{1}{2}\beta_0 E_4, \\
		\lvert 2\mu \beta_0 \Re\inner{i \partial_y v\partial_y \theta}{\partial_y^2\theta} \rvert
		&\leq \frac{\mu}{2}\alpha_0 E_2+2\mu  \frac{\beta_0^2}{\alpha_0}  \norm{\partial_y v \partial_y \theta}_2^{2} , \\
		\lvert \mu \beta_0 \Re \inner{i \partial_y^2v\theta}{\partial_y^2\theta} \rvert
		&\leq \frac{\mu}{2}\alpha_0 E_2+\frac{\mu }{2}\frac{\beta_0^2}{\alpha_0}\norm{\partial_y^2v\theta}_2^{2} , \\
		\lvert 2\mu\gamma_0\Re\inner{\partial_y v\partial_y^2 v \theta}{\partial_y\theta} \rvert &\leq \frac{\mu }{2}\gamma_0  \lVert \partial_y v \partial_y \theta \rVert_2^2 +2 \gamma_0 \mu \norm{\partial_y^2v\theta}_2^{2} .
	\end{align*}
    Together with \eqref{hyp:param_ineq} this gives us
    \begin{equation}
\begin{aligned}
\label{hyp:Phi_post_standard}
   \dv{}{\tau}\Phi&+\frac{\mu}{2}E_1+\frac{\beta_0}{4}E_4+B_7 E_7+\frac{\mu\gamma_0}{2}\norm{\partial_y v \partial_{y} \theta}_2^{2}+\gamma_1\mu \norm{v \partial_y \theta}_2^{2}\\
    \leq{} &\mu   \gamma_0 E_0
    -\varsigma\beta_0\Re \inner{i\partial_{y}^2 v \theta}{\partial_{y} \theta}-B_6 E_6 +\beta_1 2 \mu \inner{v \partial_y \theta}{\partial_y v \partial_y \theta},
\end{aligned}
\end{equation}
where we used that $\norm{\partial_y^2 v}^2_{L^\infty_y}\leq 1$. Note that the additional norms on the left-hand side can be neglected, as they are strictly positive.
 We proceed to control the three remaining inner products on the right-hand side. \\
 \\
 \emph{Control non-standard error terms.}\\
These are produced either by the cross-stream transport term in \eqref{hyp:eq} or from derivatives of $E_6, E_7$. Hence, they do not appear in the classical computations and we refer to them as non-standard.\\
\\
1. Control of $\beta_1 2 \mu \inner{v \partial_y \theta}{\partial_y v \partial_y \theta}$:\\
By employing Young's inequality and using our choice of $\beta_1$ we find 
\begin{align*}
    2\beta_1  \mu \inner{v \partial_y \theta}{\partial_y v \partial_y \theta}={}&2\frac{\beta_0^{\frac{1}{4}}}{\mu^{\frac{p}{2}}}\frac{\beta_0^{\frac{3}{4}}}{\mu^{\frac{q}{2}}} \mu \inner{v \partial_y \theta}{\partial_y v \partial_y \theta}\\
    \leq{}& \mu \frac{\beta_0^{\frac{1}{2}}}{\mu^q} \norm{v \partial_y \theta}_2^{2}+\mu \frac{\beta_0^{\frac{3}{2}}}{\mu^p} \norm{\partial_y v \partial_y \theta}_2^{2}\\
    \leq{}&\frac{\gamma_1\mu }{2}\norm{v \partial_y \theta}_2^{2}+\frac{\gamma_0\mu}{8} \norm{\partial_y v \partial_y \theta}_2^{2}.
\end{align*}
Hence \eqref{hyp:Phi_post_standard} simplifies to
\begin{equation}
\label{hyp:Phi_post_standard_1}
\begin{aligned}
   \dv{}{\tau}\Phi&+\frac{\mu}{2}E_1+\frac{\beta_0}{4}E_4+B_7E_7+\frac{\mu\gamma_0}{4}\norm{\partial_y v \partial_{y} \theta}_2^{2}+\frac{\gamma_1\mu }{2} \norm{v \partial_y \theta}_2^{2}\\
    \leq{}&\mu   \gamma_0 E_0-\varsigma\beta_0\Re \inner{i\partial_{y}^2 v \theta}{\partial_{y} \theta}-B_6E_6.
\end{aligned}
\end{equation}
2. Control of $\varsigma\beta_0\Re \inner{i\partial_{y}^2 v \theta}{\partial_{y} \theta}$:\\
First we note that since $\partial_y^2 v =- v$ we have that
\begin{align*}
    c\beta_0\Re \inner{i\partial_{y}^2 v \theta}{\partial_{y} \theta}=-\varsigma\beta_0\Re \inner{i v \theta}{\partial_{y} \theta}.
\end{align*}
From there we again use Young's inequality and the fact that we have $c=\varsigma_k\mu^\ell$ in order to obtain
\begin{align*}
    \abs{\varsigma\beta_0\Re \inner{i v \theta}{\partial_{y} \theta}}&= \beta_0\varsigma_k\mu^\ell\abs{\Re \inner{i v \theta}{\partial_{y} \theta}}\\
    &=\beta_0\varsigma_k\mu^{\frac{2\ell-p-q}{4}}\mu^{\frac{2\ell+p+q}{4}}\abs{\Re \inner{i v \theta}{\partial_{y} \theta}}\\
    &\leq \frac{1}{2}\beta_0\varsigma_k\mu^{\frac{2\ell-p-q}{2}} E_7 +\frac{1}{2}\beta_0\varsigma_k\mu^{\frac{2\ell+p+q}{2}}E_1.
\end{align*}
Hence \eqref{hyp:Phi_post_standard_1} simplifies, by expanding $B_7$ according to \eqref{hyp:B7}, to
\begin{equation}
\label{hyp:Phi_post_standard_2}
\begin{aligned}
    \dv{}{\tau}\Phi&+\frac{\mu}{2}E_1+\frac{\beta_0}{4}E_4\\
    &+\qty(\frac{1}{2}\beta_0\varsigma_k\mu^{\frac{2\ell-p-q}{2}}+\beta_0^{1/2}\mu^{1-q})E_7+\frac{\mu\gamma_0}{4}\norm{\partial_y v \partial_{y} \theta}_2^{2}+\frac{\gamma_1\mu }{2} \norm{v \partial_y \theta}_2^{2}\\
    \leq{}&\mu   \gamma_0 E_0-B_6E_6+\frac{1}{2}\beta_0\varsigma_k\mu^{\frac{2\ell+p+q}{2}}E_1.
    \end{aligned}
\end{equation}
3. Control of $B_6E_6$:\\
We recall from \eqref{hyp:B6} that, $B_6=\beta_0^{3/2}\varsigma_k\mu^{\ell-p}-\beta_0^{1/2}\varsigma_k\mu^{\ell-q}+4\beta_0\mu^{\frac{2-p-q}{2}}$.\\
We now introduce the condition 
\begin{equation}
    \label{hyp:p_q_fix}
    p > q,
\end{equation}
as then, by noting that $\beta_0$ is a constant, we can again choose $\mu_0$ so that for all $0<\mu\le \mu_0$,
\[
\abs{B_6}=\abs{\beta_0^{3/2}\varsigma_k\mu^{\ell-p}-\beta_0^{1/2}\varsigma_k\mu^{\ell-q}
+4\beta_0\mu^{\frac{2-p-q}{2}}}\le 2\beta_0^{3/2} \varsigma_k\mu^{\ell-p}+4\beta_0\mu^{\frac{2-p-q}{2}}.
\]
Then once again by Young's inequality we find
\begin{align*}
    \abs{B_6 E_6}&\leq\qty(2\beta_0^{3/2} \varsigma_k\mu^{\ell-p}+4\beta_0\mu^{\frac{2-p-q}{2}})\abs{\inner{v \theta}{\partial_y v \theta}}\\
    &\leq\qty(\frac{\beta_0}{16}+\frac{\beta_0}{16}) E_4 + \qty(16\beta_0^{2}\varsigma_k^2\mu^{2(\ell-p)}+64\beta_0 \mu^{2-p-q})E_7.
\end{align*}
In order to absorb this estimate, we need to introduce the condition 
\begin{align}
\tag{C1}
\label{hyp:mu:E_6_condition}
    \mu^{2(\ell-p)}\leq \mu^{\frac{2\ell-p-q}{2}}
\end{align}
and note that
$
    \mu^{2-p-q}< \mu^{1-p}
$
as $p,q<1$.\\
Therefore, we obtain
\begin{align*}
    \abs{B_6 E_6}
    &\leq\frac{\beta_0}{8} E_4 + \qty(16\beta_0^{2}\varsigma_k^2\mu^{\frac{2\ell-p-q}{2}}+64\beta_0 \mu^{1-q})E_7.
\end{align*}
We now take $\beta_0$ small s.t.
\begin{align}
\label{hyp:beta_0_conditon_error}
64 \beta_0^{1/2} \le \frac12
\iff \beta_0 \le \frac{1}{128^2},
\end{align}
and note that $\varsigma_k^2\leq \varsigma_k$ (cf. \eqref{hyp:c_0}), this gives us the estimate
\begin{align*}
    \abs{B_6 E_6}
    &\leq\frac{\beta_0}{8} E_4 + \qty(\frac{\beta_0}{4}\varsigma_k\mu^{\frac{2\ell-p-q}{2}}+\frac{\beta_0^{1/2}}{2} \mu^{1-q})E_7.
\end{align*}
Then we can bound the $E_6$-term in \eqref{hyp:Phi_post_standard_2} to finally get
\begin{equation*}
\begin{aligned}
   \dv{}{\tau}\Phi&+\frac{\mu}{2}E_1+\frac{\beta_0}{8}E_4+\qty(\frac{1}{4}\beta_0\varsigma_k\mu^{\frac{2\ell-p-q}{2}}+\frac{1}{2}\beta_0^{1/2}\mu^{1-q}) E_7\\
   &+\frac{\mu\gamma_0}{4}\norm{\partial_y v \partial_{y} \theta}_2^{2}+\frac{\gamma_1\mu }{2} \norm{v \partial_y \theta}_2^{2}\\
    \leq{}&\mu   \gamma_0 E_0+\frac{1}{2}\beta_0^{3}\mu^{\frac{2\ell+p+q}{2}}E_1.
\end{aligned}
\end{equation*}
Which, by slightly relaxing the constant in front of $E_7$, we can simplify to
\begin{equation}
\label{hyp:Phi_post_non_standard}
\begin{aligned}
   \dv{}{\tau}\Phi&+\frac{\mu}{2}E_1+\frac{\beta_0}{8}E_4+\frac{1}{4}B_7E_7+\frac{\mu\gamma_0}{4}\norm{\partial_y v \partial_{y} \theta}_2^{2}+\frac{\gamma_1\mu }{2} \norm{v \partial_y \theta}_2^{2}\\
    \leq{}&\mu   \gamma_0 E_0+\frac{1}{2}\beta_0^{3}\mu^{\frac{2\ell+p+q}{2}}E_1.
\end{aligned}
\end{equation}
We have now successfully controlled all error terms. Therefore we now need to control norms that appear with the wrong sign, namely $E_0, E_1$ on the right-hand side. \\
\\
\emph{Flow specific optimisations.}\\
We will now leverage the fact that we have a specific expression for our flow as $v(y)=\sin(y)$. This allows us to create an $E_1$-term with an improved scaling in $\mu$.
We note that
\begin{align*}
    v^2+(\partial_y v)^2=1,
\end{align*}
leading to the relation
\begin{align}
    E_1= \norm{\partial_y v \partial_{y} \theta}_2^{2}+\norm{v \partial_y \theta}_2^{2}.
\end{align}
Recall that $p\geq q$ (cf. \eqref{hyp:p_q_fix}) and that $\beta_0\leq \frac{1}{8}$ (cf. \eqref{hyp:beta_0_conditon_error}), then we have 
\begin{align*}
    \frac{\mu\gamma_0}{4}\norm{\partial_y v \partial_{y} \theta}_2^{2}+\frac{\gamma_1\mu }{2}\norm{v \partial_y \theta}_2^{2}&=\frac{16 \beta_0^{3/2}\mu^{1-p}}{4}\norm{\partial_y v \partial_{y} \theta}_2^{2}+\frac{\beta_0^{1/2}\mu^{1-q}}{2}\norm{v \partial_y \theta}_2^{2}\\
    &\geq 4\beta_0^{3/2}\mu^{1-q}\qty[\norm{\partial_y v \partial_{y} \theta}_2^{2}+\norm{v \partial_y \theta}_2^{2}]\\
    &= 4\beta_0^{3/2}\mu^{1-q} E_1.
\end{align*}
Hence we can rewrite \eqref{hyp:Phi_post_non_standard} as
\begin{align}
   \dv{}{\tau}\Phi&+\qty(\frac{\mu}{2}+4\beta_0^{3/2}\mu^{1-q})E_1+\frac{\beta_0 }{8}E_4+\frac{1}{4}B_7E_7\\
    \leq{}&\mu   \gamma_0 E_0+\frac{1}{2}\beta_0^{3}\mu^{\frac{2\ell+p+q}{2}}E_1.
\end{align}
Now we introduce the additional constraint
\begin{align}
\tag{C2}
    \label{hyp:mu:E_1_condition}
    \mu^{\frac{2\ell+p+q}{2}}\leq \mu^{1-q}.
\end{align}
Then we have 
\begin{align*}
   &\dv{}{\tau}\Phi+\qty(\frac{\mu}{2}+3\beta_0^{3/2}\mu^{1-q})E_1+\frac{\beta_0 }{8}E_4+B_7 E_7
    \leq\mu   \gamma_0 E_0.
\end{align*}
Since $q>0$ the leading order coefficient in $\mu$ in front of $E_1$ is $\mu^{1-q}$. Therefore we may drop the comparatively very small term of $\frac{\mu}{2}$. Similarly we recall that $B_7=\beta_0 \varsigma_k\mu^{\frac{2\ell-p-q}{2}}
     +\beta_0^{1/2}\mu^{1-q}$ and that in our regime of $0<q<p<1$ and $0<\ell<1$ we always have
     \begin{align*}
         \mu^{\frac{2\ell-p-q}{2}}>\mu^{1-q}.
     \end{align*}
Hence the leading order coefficient in $\mu$ in $B_7$ is determined by it's first term and we can safely drop the second. We obtain
\begin{align}
\label{hyp:pre_spectral_final}
   &\dv{}{\tau}\Phi+3\beta_0^{3/2}\mu^{1-q}E_1+\frac{\beta_0 }{8}E_4+\frac{\beta_0}{4}\varsigma_k\mu^{\frac{2\ell-p-q}{2}} E_7
    \leq\mu   \gamma_0 E_0.
\end{align}
\paragraph{Step 3: Spectral estimate.}
Through the above procedure, we have now improved the $\mu$-scaling of $E_1$ and controlled the respective term on the right. This leaves us to address the $E_0$-term on the right, while also noting that we require $E_0$ on the left.
We will do this via a spectral gap inequality provided by Lemma \ref{hyp:L:spectral_gap}, which we recall as
	\begin{equation}
		\sigma^{\frac{1}{2}} E_0 \lesssim C_{s} [\sigma E_1 + E_4]\,.
	\end{equation}
 In order to apply this to \eqref{hyp:pre_spectral_final} we observe that
 \begin{align}
 \label{hyp:SG}
   2\beta_0^{3/2}\mu^{1-q}E_1+\frac{\beta_0 }{16}E_4=\frac{\beta_0 }{16}\qty[32 \beta_0^{1/2}\mu^{1-q}E_1+E_4]\gtrsim \frac{\beta_0^{5/4}}{C_{s}} \mu^{(1-q)/2}E_0,
 \end{align}
 where we choose $\sigma =32 \beta_0^{1/2}\mu^{1-q}$.\\
 Now in order to absorb the right-hand side into this term we need that
 \begin{align}
 \label{hyp:SG_helper}
     \mu \gamma_0= 16\beta_0^{3/2}\mu^{1-p}\leq \frac{\beta_0^{5/4}}{2C_{s}} \mu^{(1-q)/2}.
 \end{align}
 This can only be achieved under the condition that
 \begin{align}
 \tag{C3}
     \label{hyp:mu:E_0_condition}
     \mu^{(1-q)/2}\geq \mu^{1-p}.
 \end{align}
 We keep this as a further constraint and assume it to be true for now. We then need 
 \begin{align*}
     16\beta_0^{3/2}\mu^{1-p} \leq\frac{\beta_0^{5/4}}{2C_{s}} \mu^{1-p} 
     \implies\beta_0^{1/4}\leq \frac{1}{32 C_{s}},
 \end{align*}
 which can simply be achieved by picking 
 \begin{align}
     \label{hyp:beta_0_condition_spectral}
     \beta_0\leq \frac{1}{32^4 C_{s}^4}.
 \end{align}
 Then we can indeed apply \eqref{hyp:SG} to \eqref{hyp:pre_spectral_final} in order to obtain
 \begin{align*}
    &\dv{}{\tau}\Phi+ \frac{\beta_0^{5/4}}{C_{s}} \mu^{(1-q)/2}E_0+\beta_0^{3/2}\mu^{1-q}E_1+\frac{\beta_0}{4}\varsigma_k\mu^{\frac{2\ell-p-q}{2}} E_7
    \leq \mu   \gamma_0 E_0
 \end{align*}
 and by \eqref{hyp:SG_helper} we can now absorb the right-hand side into the left to get
 \begin{align}\label{hyp:post_spectral}
    &\dv{}{\tau}\Phi+ \frac{\beta_0^{5/4}}{2C_{s}} \mu^{(1-q)/2}E_0+\beta_0^{3/2}\mu^{1-q}E_1+\frac{\beta_0 }{16}E_4+\frac{\beta_0}{4}\varsigma_k\mu^{\frac{2\ell-p-q}{2}} E_7
    \leq 0.
 \end{align}
\paragraph{Step 3: Final Optimisation.}
In view of \eqref{hyp:post_spectral}, the coercivity estimate \eqref{hyp:coercivity} can be used to derive the desired bound for $\Phi$. This requires an optimisation of the parameters subject to the constraints imposed throughout the proof.\\
As detailed in Section~\ref{hyp:A:optimisation} of the Appendix, \eqref{hyp:post_spectral} can be rewritten as
\begin{align} &\dv{}{\tau}\Phi+ 
    \frac{\beta_0^{5/4}}{12C_{s}}\varsigma_k\mu^{p}\qty[ 4E_0 + 5\alpha_0 E_1+ 6\gamma_0 E_4 +5\gamma_1 E_7]\leq 0. 
\end{align} 
The admissible range of $\ell$ and explicit formulas for $p,q$ are then given by \eqref{hyp:y_final_range} as 
\[ \ell\in\Big(\frac13,\frac34\Big)\ ,\qquad p=\frac{1+2\ell}{5},\qquad q=\frac{3-4\ell}{5}. \]
Lastly, we use the coercivity of the functional \eqref{hyp:coercivity} to immediately get
\begin{align} &\dv{}{\tau}\Phi+  \frac{\beta_0^{5/4}}{12C_{s}}\varsigma_k\mu^{p}\Phi\leq 0. 
\end{align}

\section{Large translation speeds}
\label{sec:lc}
In this section we prove Theorem~\ref{lc:thm:large_c}. To this end, we compare the solution of the advection--diffusion equation~\eqref{adv}, which we recall here as
\begin{equation}\label{lc:eq:f} 
    \partial_t \Theta + \sin(y-ct)\,\partial_x \Theta = \nu \Delta \Theta,
\end{equation}
with the solution $\Theta_H$ of the corresponding heat equation
\begin{equation}\label{lc:eq:fH} 
    \partial_t \Theta_H = \nu \Delta \Theta_H,
\end{equation}
subject to the same initial datum $\Theta(\cdot,0) = \Theta_H(\cdot,0) = \Theta_0$. 

We introduce the deviation $w := \Theta - \Theta_H$, which, by linearity, satisfies the forced diffusion equation
\begin{equation}\label{lc:eq:w} 
    \partial_t w - \nu \Delta w = -\sin(y-ct)\,\partial_x \Theta, 
    \qquad w(\cdot,0) = 0.
\end{equation}

\subsection{Proof of Theorem \ref{lc:thm:large_c}}
Testing \eqref{lc:eq:w} with $w$ and integrating over $\T^2\times(0,t)$ yields
\begin{equation}\label{lc:eq:energy-id-int}
    \frac12\|w\|_2^2+\nu\int_0^t \|\nabla w\|_2^2ds
    =-\int_0^t\int_{\T^2}\sin(y-cs)\,\partial_x \Theta\,w\,dx\,dy\, ds.
\end{equation}
We now need to control the right-hand side of the above equation, for which we recall that $\partial_s\cos(y-cs)= c \sin(y-cs)$ and hence via integration by parts we find
\begin{align*}
-\int_0^t&\int_{\T^2}\sin(y-cs)\,\partial_x \Theta\,w\,dx\,dy\, ds\\
=&-\frac{1}{c}\int_0^t\int_{\T^2}\partial_s\cos(y-cs)\,\partial_x \Theta\,w\,dx\,dy\, ds\\
=&\frac{1}{c}\int_0^t\int_{\T^2}\cos(y-cs)\partial_s(\partial_x \Theta\,w )\,dx\,dy\, ds\\
&-\frac{1}{c}\int_{\T^2}\cos(y-ct)\,\partial_x \Theta\,w\,dx\,dy\\
=&\frac{1}{c}\int_0^t\int_{\T^2}\cos(y-cs)(\partial_x\partial_s \Theta\,w +\partial_x \Theta\,\partial_sw )\,dx\,dy\, ds\\
&-\frac{1}{c}\int_{\T^2}\cos(y-ct)\,\partial_x \Theta\,w\,dx\,dy.
\end{align*}

Differentiating \eqref{lc:eq:f} with respect to $x$ we have
\[
\partial_x\partial_s \Theta+\sin(y-cs)\,\partial_x^2 \Theta=\nu\Delta \partial_x \Theta.
\]
Thus, using \eqref{lc:eq:w}, we find
\begin{align*}
    \frac{1}{c}\int_0^t&\int_{\T^2}\cos(y-cs)(\partial_x\partial_s \Theta\,w +\partial_x \Theta\,\partial_sw )\,dx\,dy\, ds\\
    &-\frac{1}{c}\int_{\T^2}\cos(y-ct)\,\partial_x \Theta\,w\,dx\,dy.\\
    =&\frac{1}{c}\int_0^t\int_{\T^2}\cos(y-cs)\qty(-\sin(y-cs )\partial_x^2 \Theta +\nu \Delta \partial_x \Theta)w\,dx\,dy\, ds\\
    &+\frac{1}{c}\int_0^t\int_{\T^2}\cos(y-cs)\qty(-\sin(y-cs )\partial_x \Theta +\nu \Delta w)\partial_x \Theta\,dx\,dy\, ds\\
    &-\frac{1}{c}\int_{\T^2}\cos(y-ct)\,\partial_x \Theta\,w\,dx\,dy.\\
    \leq{}& 
    -\frac{\nu}{c}\int_0^t\int_{\T^2}\nabla(\cos(y-cs)w)\cdot \nabla \partial_x \Theta\,dx\,dy\, ds\tag{I}\\
    &-\frac{\nu}{c}\int_0^t\int_{\T^2}\nabla(\cos(y-cs)\partial_x \Theta) \cdot\nabla w\,dx\,dy\, ds\tag{II}\\
    &+\frac{1}{c}\int_0^t\int_{\T^2}|\partial_x^2 \Theta ||w|\,dx\,dy\, ds\tag{III}\\
    &+\frac{1}{c}\int_0^t\int_{\T^2}|\partial_x \Theta |^2\,dx\,dy\, ds\tag{IV}\\
    &-\frac{1}{c}\int_{\T^2}\cos(y-ct)\,\partial_x \Theta\,w\,dx\,dy.
    \tag{BT}
\end{align*}
Let us estimate the terms (I)-(III) and the boundary term (BT) individually.\\
\\
For (I) we note that
\begin{align*}
\nabla&\big(\cos(y-cs)\,w\big)\cdot \nabla(\partial_x \Theta)
=
\begin{pmatrix}
\partial_x\big(\cos(y-cs)\,w\big) \\
\partial_y\big(\cos(y-cs)\,w\big)
\end{pmatrix}
\cdot
\begin{pmatrix}
\partial_x^2 \Theta \\
\partial_x\partial_y \Theta
\end{pmatrix}
\\[4pt]
&\quad= -\sin(y-cs) w \partial_x\partial_y \Theta+\cos(y-cs)\,\partial_y w\partial_x\partial_y \Theta+
\cos(y-cs)\,\partial_x w\,\partial_x^2 \Theta,
\end{align*}
and 
\begin{align*}
\nabla&\big(\cos(y-cs)\,\partial_x \Theta\big)\cdot \nabla w
=
\begin{pmatrix}
\partial_x\big(\cos(y-cs)\,\partial_x \Theta\big)\\[2pt]
\partial_y\big(\cos(y-cs)\,\partial_x \Theta\big)
\end{pmatrix}
\cdot
\begin{pmatrix}
\partial_x w\\[2pt]
\partial_y w
\end{pmatrix}
\\[6pt]
&\quad=-\sin(y-cs)\,\partial_x \Theta \partial_y w+ \cos(y-cs)\,\partial_x\partial_y \Theta\partial_y w
+\cos(y-cs)\,\partial_x^2 \Theta\,\partial_x w.
\end{align*}
Now by using the Cauchy-Schwarz and Young's inequalities we find
\begin{align*}
   |(I)|\leq{}& \frac{\nu}{c} \int_0^t\int_{\T^2}|w||\partial_x\partial_y\Theta|\,dx\,dy\, ds
   +\frac{\nu}{c} \int_0^t\int_{\T^2}|\partial_y w||\partial_x\partial_y\Theta|\,dx\,dy\, ds\\
   &+\frac{\nu}{c} \int_0^t\int_{\T^2}|\partial_x w||\partial_{x}^2\Theta|\,dx\,dy\, ds\\
  \leq{}& \frac{\nu}{2c}\int_0^t \norm{w}^2_2 \,ds+\frac{\nu}{2c}\int_0^t \norm{\partial_x\partial_y\Theta}^2_2 \, ds\\
  &+\frac{\nu}{2c}\int_0^t \norm{\partial_y w}^2_2 \,ds+\frac{\nu}{2c}\int_0^t \norm{\partial_x\partial_y\Theta}^2_2 \, ds\\
  &+\frac{\nu}{2c}\int_0^t \norm{\partial_x w}^2_2 \,ds+\frac{\nu}{2c}\int_0^t \norm{\partial_{x}^2\Theta}^2_2 \, ds,\\
     |(II)|\leq{}& \frac{\nu}{c} \int_0^t\int_{\T^2}|\partial_y w||\partial_{x}\Theta|\,dx\,dy\, ds
   +\frac{\nu}{c} \int_0^t\int_{\T^2}|\partial_y w||\partial_x\partial_y\Theta|\,dx\,dy\, ds\\&+\frac{\nu}{c} \int_0^t\int_{\T^2}|\partial_x w||\partial_{x}^2\Theta|\,dx\,dy\, ds\\
   \leq{}&\frac{\nu}{2c}\int_0^t \norm{\partial_y w}^2_2 \,ds+\frac{\nu}{2c}\int_0^t \norm{\partial_{x}\Theta}^2_2 \, ds\\&
   +\frac{\nu}{2c}\int_0^t \norm{\partial_y w}^2_2 \,ds+\frac{\nu}{2c}\int_0^t \norm{\partial_x\partial_y\Theta}^2_2 \, ds\\
   &+\frac{\nu}{2c}\int_0^t \norm{\partial_x w}^2_2 \,ds+\frac{\nu}{2c}\int_0^t \norm{\partial_{x}^2\Theta}^2_2 \, ds,\\
    |(III)|\leq{}& \frac{1}{c}\qty(\int_0^t\int_{\T^2}|w|^2\,dx\,dy\, ds)^{1/2}\qty(\int_0^t\int_{\T^2}|\partial_x^2 \Theta|^2\,dx\,dy\, ds)^{1/2}\\
    ={}&\frac{1}{c^{1/2}}\qty(\int_0^t\int_{\T^2}|w|^2\,dx\,dy\, ds)^{1/2}\frac{1}{c^{1/2}}\qty(\int_0^t\int_{\T^2}|\partial_x^2 \Theta|^2\,dx\,dy\, ds)^{1/2}\\
    \leq{}& \frac{1}{2c}\int_0^t\int_{\T^2}|w|^2\,dx\,dy\, ds+\frac{1}{2c}\int_0^t\int_{\T^2}|\partial_x^2 \Theta|^2\,dx\,dy\, ds,\\
|\mathrm{BT}|
\le& \frac{1}{c}\,\|\partial_x \Theta\|_{2}\,\|w\|_{2}\\
\le& \frac{1}{2c}\|w\|_2^2 + \frac{1}{2c}\|\partial_x \Theta\|_2^2.
\end{align*}
All together then gives us an estimate for the RHS of \eqref{lc:eq:energy-id-int} and the combined estimate
\begin{align*}
     \frac12\|w\|_2^2&+\nu\int_0^t \|\nabla w\|_2^2ds\\
    \leq{}& \frac{1}{2c}\int_0^t\norm{w}^2_2\, ds+\frac{1}{2c}\int_0^t\norm{\partial_x^2 \Theta}^2_2\, ds
    +\frac{1}{c}\int_0^t\norm{\partial_x \Theta}^2_2\, ds\\
    &+\frac{\nu}{2c}\int_0^t \norm{w}^2_2 \,ds+\frac{\nu}{2c}\int_0^t \norm{\partial_x\partial_y\Theta}^2_2 \, ds
  +\frac{\nu}{2c}\int_0^t \norm{\partial_y w}^2_2 \,ds\\&+\frac{\nu}{2c}\int_0^t \norm{\partial_x\partial_y\Theta}^2_2 \, ds
  +\frac{\nu}{2c}\int_0^t \norm{\partial_y w}^2_2 \,ds+\frac{\nu}{2c}\int_0^t \norm{\partial_{x}\Theta}^2_2 \, ds\\
&   +\frac{\nu}{2c}\int_0^t \norm{\partial_y w}^2_2 \,ds+\frac{\nu}{2c}\int_0^t \norm{\partial_x\partial_y\Theta}^2_2 \, ds
   +\frac{\nu}{c}\int_0^t \norm{\partial_x w}^2_2 \,ds\\&+\frac{\nu}{c}\int_0^t \norm{\partial_{x}^2\Theta}^2_2 \, ds+\frac{1}{2c}\|w\|_2^2 + \frac{1}{2c}\|\partial_x \Theta\|_2^2.
\end{align*}
Now using the fact that $\norm{\partial_y w}^2_2 \leq \norm{\nabla w}^2_2,\, \norm{\partial_x w}^2_2 \leq \norm{\nabla w}^2_2 $ we can write the above estimate as
\begin{align*}
   &\qty(\frac12-\frac{1}{2c})\|w\|_2^2+\nu\qty(1-\frac{5}{2}\frac{1}{c})\int_0^t \|\nabla w\|_2^2ds\\
   &\quad \leq \frac{\nu}{2c}\int_0^t \norm{w}^2_2 \,ds+\frac{1}{2c}\int_0^t\norm{w}^2_2\, ds+\frac{1}{2c}\int_0^t\norm{\partial_x^2 \Theta}^2_2\, ds\\
    &\qquad+\frac{1}{c}\int_0^t\norm{\partial_x \Theta}^2_2\, ds + \frac{1}{2c}\|\partial_x \Theta\|_2^2\\
    &\qquad+C_1 \frac{\nu}{c}\int_0^t \qty(\norm{\partial_x\partial_y\Theta}^2_2 +\norm{\partial_{x}^2\Theta}^2_2 +\norm{\partial_{x}\Theta}^2_2) \, ds.
\end{align*}
Since we are interested in the large $c$ regime we can safely assume that
\begin{align*}
    c>\frac{5}{2},
\end{align*} 
which guarantees that the second term is positive and we can therefore drop it.\\
Hence we can write the overall estimate as
\begin{align}
    \label{lc:to-gronwall}
    B \|w\|_2^2 \leq A \int_0^t \norm{w}^2_2 \,ds+  F_{\nu}(t),
\end{align}
where
\begin{align}
\label{lc:A_B}
    B=&\qty(\frac12-\frac{1}{2c}), \quad A=\qty(\frac{\nu}{2c}+\frac{1}{2c}),\\
    F_{\nu}(t)&=\frac{1}{c}\int_0^t\norm{\partial_x \Theta}^2_2\, ds + \frac{1}{2c}\|\partial_x \Theta\|_2^2\\
    &+C_1 \frac{\nu}{c}\int_0^t\qty( \norm{\partial_x\partial_y\Theta}^2_2 +\norm{\partial_{x}^2\Theta}^2_2 +\norm{\partial_{x}\Theta}^2_2 )\, ds.
\end{align}

We split the estimate of $F_{\nu}(t)$ in several steps.
\begin{enumerate}
    \item For $\displaystyle\int_0^t\norm{ \partial_{x}\Theta}^2_2 \, ds$:\\
    Testing the PDE \eqref{lc:eq:f} with $\partial_x^2 \Theta$ and integrating over $\T^2$ yields
    \begin{align*}
           \int_{\T^2}\partial_t \Theta\,\partial_x^2 \Theta\,dx\,dy
    ={}&\nu\int_{\T^2}\Delta f\,\partial_x^2 \Theta\,dx\,dy-\int_{\T^2}\sin(y-ct)\partial_x \Theta\,\partial_x^2 \Theta\,dx\,dy,
    \end{align*}
    that is 
    \begin{align*}
    -&\int_{\T^2}\partial_t \partial_x \Theta\,\partial_x \Theta\,dx\,dy\\
    =&\nu\int_{\T^2}[(\partial_x^2 \Theta)^2+\partial_x^2 f\,\partial_y^2 \Theta+]\,dx\,dy-\int_{\T^2}\sin(y-ct)\frac{\partial_x (\partial_x \Theta)^2}{2}\,dx\,dy.
    \end{align*}
    Noting that the second term on the right vanishes by periodicity, we get
    \begin{align*}
       \frac12 \dv{}{t}\int_{\T^2} |\partial_x \Theta|^2\,dx\,dy
    ={}&-\nu\int_{\T^2}[(\partial_x^2 \Theta)^2+\partial_x^2 \Theta\,\partial_y^2 \Theta]\,dx\,dy.
    \end{align*}
    Using the  Young's inequality
     \begin{align}
     \label{lc:es:1}
        \frac{1}{2}\dv{}{t}\norm{\partial_x \Theta}^2_2
\leq  \frac32 \nu\norm{\partial_x^2 \Theta}^2_2+ \frac12\nu\norm{\partial_y^2 \Theta}^2_2.
    \end{align}
 \item For $\displaystyle\int_0^t\norm{ \partial_{x}^2 \Theta}^2_2 \, ds$:\\
 We start by differentiating the PDE \eqref{lc:eq:f} twice w.r.t. to $x$ in order to get
 \begin{align*}
     \partial_t \partial_x^2 \Theta=\nu\Delta \partial_x^2 \Theta-\sin(y-ct)\,\partial_x^3 \Theta.
 \end{align*}
  Testing this with $\partial_x^2 \Theta$ and integrating over $\T^2$ yields
  \begin{align*}
       \int_{\T^2}&\partial_t \,\partial_x^2 \Theta\,\partial_x^2 \Theta\,dx\,dy\\
    &=\nu\int_{\T^2}\Delta \partial_x^2 \Theta\,\partial_x^2 \Theta\,dx\,dy-\int_{\T^2}\sin(y-ct)\partial_x^3 \Theta\,\partial_x^2 \Theta\,dx\,dy.
  \end{align*}
  Thanks to the periodicity
        \begin{align}
            \frac12 \dv{t}  \|\partial_x^2 \Theta\|^2_2={}&-\nu 
\|\partial_y \partial_x^2 \Theta\|^2_2-\nu \|\partial_x^3 \Theta\|^2_2\le 0.\label{lc:es:2}
  \end{align}
  \item For $\displaystyle\int_0^t \norm{\partial_x\partial_y\Theta}^2_2\, ds$: \\
 Differentiating  \eqref{lc:eq:f} once w.r.t. to $y$ and $x$ we get
 \begin{align*}
     \partial_t \partial_x\partial_y \Theta={}&\nu\Delta \partial_x\partial_y \Theta-\partial_y(\sin(y-ct)\,\partial_x^2 \Theta)\\
    ={}&\nu\Delta \partial_x\partial_y \Theta-\cos(y-ct)\,\partial_x^2 \Theta-\sin(y-ct)\,\partial_y\partial_x^2 \Theta.
 \end{align*}
  Testing this with $\partial_x\partial_y \Theta$ and integrating over $\T^2$, integtegrating by parts, and using the periodicity   yield

  \begin{align*}
       \frac12 \dv{}{t} \int_{\T^2}|\,\partial_x\partial_y \Theta|^2\,dx\,dy
    ={}&\nu\int_{\T^2} \partial_{x}^3\partial_y \Theta\,\partial_x\partial_y \Theta\,dx\,dy+\nu\int_{\T^2} \partial_{x}\partial_y^3 \Theta\,\partial_x\partial_y \Theta\,dx\,dy\\
    &-\int_{\T^2}\cos(y-ct)\partial_x^2 \Theta\,\partial_x\partial_y \Theta\,dx\,dy\\
    &-\int_{\T^2}\sin(y-ct)\partial_x \frac{(\partial_x\partial_y \Theta)^2}{2}\,dx\,dy\\
    ={}&-\nu\int_{\T^2} |\partial_x^2\partial_y \Theta|^2\,dx\,dy-\nu\int_{\T^2} |\partial_x\partial_y^2\Theta|^2\,dx\,dy\\
    &-\int_{\T^2}\cos(y-ct)\partial_x^2 \Theta\,\partial_x\partial_y \Theta\,dx\,dy.
  \end{align*}
  By bounding $\cos(y-ct)\leq 1$ and applying Young's inequality,
  \begin{align*}
       \frac12 \dv{}{t} \int_{\T^2}|\,\partial_x\partial_y \Theta|^2\,dx\,dy\le{}&-\nu\int_{\T^2} |\partial_x^2\partial_y \Theta|^2\,dx\,dy-\nu\int_{\T^2} |\partial_x\partial_y^2\Theta|^2\,dx\,dy\\
    &+\frac12 \int_{\T^2}|\partial_x^2 \Theta|^2\,dx\,dy+\frac12 \int_{\T^2}|\partial_x\partial_y \Theta|^2\,dx\,dy,
  \end{align*}
and then  
\begin{align}
  \label{lc:es:3}
       \frac12 \dv{}{t} \|\,\partial_x\partial_y \Theta\|_2^2\le \frac12 \|\partial_x^2 \Theta\|_2^2+\frac12 \|\partial_x\partial_y \Theta\|^2_2.
  \end{align}
  \item Lastly we need to control $\norm{ \partial_{y}^2\Theta}^2_2$. 
  Differentiating \eqref{lc:eq:f} twice w.r.t. to $y$ 
 \begin{align*}
     \partial_t \partial_y^2 \Theta={}&\nu\Delta \partial_y^2 x-\partial_y^2(\sin(y-ct)\,\partial_x \Theta)\\
     ={}&\nu \partial_y^4 \Theta+\nu \partial_x^2\partial_y^2 \Theta
     +\sin(y-ct)\,\partial_x \Theta\\
     &-2\cos(y-ct)\,\partial_x\partial_y \Theta-\sin(y-ct)\,\partial_y^2 \partial_x \Theta.
 \end{align*}
  Testing this with $\partial_y^2 \Theta$, integrating over $\T^2$,  and integrating by parts yields
  \begin{align*}
       \frac12 \dv{}{t}&\int_{\T^2}|\partial_y^2 \Theta|^2\,dx\,dy\\
    =&-\nu\int_{\T^2}|\partial_y^3 \Theta|^2 \,dx\,dy
    -\nu\int_{\T^2} |\partial_y^2\partial_x \Theta|^2\,dx\,dy\\
    &+\int_{\T^2}\sin(y-ct)\partial_x \Theta\,\partial_y^2 \Theta\,dx\,dy-2\int_{\T^2}\cos(y-ct)\partial_x\partial_y \Theta\,\partial_y^2 \Theta\,dx\,dy\\
    &-\int_{\T^2}\sin(y-ct)\partial_x \frac{(\partial_{y}^2 \Theta)^2}{2}\,dx\,dy.
  \end{align*}
  The last term again vanishes by periodicity and the terms scaling with $\nu$ carry a negative sign can therefore be dropped. Then by employing Young's inequality and again bounding $\cos(y-ct)\leq 1, \,\sin(y-ct)\leq 1$ on the leftover terms we obtain
  \begin{align}
  \label{lc:es:4}
       \frac12 \dv{t} \|\partial_y^2 \Theta\|^2_2
    \le \frac12 \norm{\partial_x \Theta}^2+\frac32 \norm{\partial_y^2 \Theta}^2+\norm{\partial_x\partial_y \Theta}^2.
  \end{align}
\end{enumerate}
Now combining all four estimates \eqref{lc:es:1}, \eqref{lc:es:2}, \eqref{lc:es:3}, \eqref{lc:es:4} we get 
\begin{equation}
\label{eq:GMC}
\begin{split}
    \frac12 \dv{t} &\qty( \|\partial_x \Theta\|^2_2
    +\|\partial_x^2 \Theta\|^2_2
    +\|\partial_x\partial_y \Theta\|^2_2
    +\|\partial_y^2 \Theta\|^2_2) \\
    \leq{}& 
    \frac{3\nu+1}{2}\,\|\partial_x^2 \Theta\|_2^2
    + \frac{\nu+3}{2}\,\|\partial_y^2 \Theta\|_2^2
    \\&+ \frac{3}{2}\,\|\partial_x\partial_y \Theta\|_2^2
    + \frac{1}{2}\,\|\partial_x \Theta\|_2^2.
\end{split}
\end{equation}
Setting
\begin{align*}
    \Psi(t)
    :=\|\partial_x \Theta\|^2_2
     +\|\partial_x^2 \Theta\|^2_2
     +\|\partial_x\partial_y \Theta\|^2_2
     +\|\partial_y^2 \Theta\|^2_2,
\end{align*}
\eqref{eq:GMC} reads
\begin{align*}
    \frac12 \Psi'(t)
    \le{}& \frac{3\nu+1}{2}\,\|\partial_x^2 \Theta\|_2^2
    + \frac{\nu+3}{2}\,\|\partial_y^2 \Theta\|_2^2
    \\&+ \frac{3}{2}\,\|\partial_x\partial_y \Theta\|_2^2
    + \frac{1}{2}\,\|\partial_x \Theta\|_2^2.
\end{align*}
For $0<\nu\le 1$ we have
\[
\frac{3\nu+1}{2} \le 2,
\qquad
\frac{\nu+3}{2} \le 2,
\qquad
\frac{3}{2} \le 2,
\qquad
\frac{1}{2} \le 2,
\]
and hence
\[
\frac12 \Psi'(t)
\le 
2\,\Psi(t),
\]
that gives
\begin{equation*}
    \Psi(t)\leq  e^{4 t}\Psi_0,
\end{equation*}
where $\Psi_0:=\Psi(0)$.

\paragraph{Bounding $w$.}
Now we go back to \eqref{lc:to-gronwall} and recall it as
\begin{align*}
    B \|w\|_2^2 \leq A \int_0^t \norm{w}^2_2 \,ds+  F_{\nu}(t).
\end{align*}
and that
\begin{align*}
     F_{\nu}(t)={}&\frac{1}{c}\int_0^t\norm{\partial_x \Theta}^2_2\, ds + \frac{1}{2c}\|\partial_x \Theta\|_2^2\\&
    +C_1 \frac{\nu}{c}\int_0^t\qty( \norm{\partial_x\partial_y\Theta}^2_2 +\norm{\partial_{x}^2\Theta}^2_2 +\norm{\partial_{x}\Theta}^2_2) \, ds\\
    \leq{}&\frac{1}{c}\int_0^t\norm{\partial_x \Theta}^2_2\, ds + \frac{1}{2c}\|\partial_x \Theta\|_2^2+C_1 \frac{\nu}{c}\int_0^t \Psi(s) \, ds\\
    \leq{}&\frac{1}{c}\int_0^t\norm{\partial_x \Theta}^2_2\, ds + \frac{1}{2c}\|\partial_x \Theta\|_2^2+C_1 \frac{\nu}{c}\int_0^te^{4 s}\Psi_0 \, ds.
\end{align*}
Since 
\begin{align*}
     \Psi(t)={}&\|\partial_x \Theta\|^2_2
     +\|\partial_x^2 \Theta\|^2_2+\|\partial_x\partial_y \Theta\|^2_2+\|\partial_y^2 \Theta\|^2_2\\
     \leq{}& \|\nabla^2 \Theta\|^2_2+\|\partial_x \Theta\|^2_2,
\end{align*}
we obtain
\begin{align*}
     F_{\nu}(t)\le{}&\frac{1}{c}\int_0^t\norm{\partial_x \Theta}^2_2\, ds+ \frac{1}{2c}\|\partial_x \Theta\|_2^2\\&+C_1 \frac{\nu}{c}\int_0^t e^{4 s}\qty(\|\nabla^2 \Theta_0\|^2_2+\|\partial_x \Theta_0\|^2_2 \,) ds\\
     \le{}&\frac{1}{c}\int_0^t\norm{\partial_x \Theta}^2_2\, ds + \frac{1}{2c}\|\partial_x \Theta\|_2^2\\&+\frac{C_1}{4} \frac{\nu}{c}\qty(\|\nabla^2 \Theta_0\|^2_2+\|\partial_x \Theta_0\|^2_2) (e^{4 t}-1).
\end{align*}
So overall we have
\begin{align}
\label{lc:pre-gronwall}
    B \|w\|_2^2 \leq{}& A \int_0^t \norm{w}^2_2 \,ds+  \frac{1}{c}\int_0^t\norm{\partial_x \Theta}^2_2\, ds + \frac{1}{2c}\|\partial_x \Theta\|_2^2\\&+\frac{C_1}{4}\frac{\nu}{c}\qty(\|\nabla^2 \Theta_0\|^2_2+\|\partial_x \Theta_0\|^2_2) (e^{4 t}-1).
    \notag
\end{align}
\paragraph{Control of the remaining derivatives}
We now need to control the last remaining terms on the right–hand side involving $\partial_x \Theta$. From the definition of $\Psi$ we immediately have
\[
\|\partial_x \Theta\|_2^2 \le \Psi(t), \qquad
\int_0^t \|\partial_x \Theta\|_2^2\,ds \le \int_0^t \Psi(s)\,ds.
\]
Using the bound $\Psi(s)\le e^{4 s}\Psi_0$ for $0\le s\le t$ we obtain
\begin{align*}
\int_0^t \|\partial_x \Theta\|_2^2\,ds
&\le \Psi_0 \int_0^t e^{4 s}\,ds
= \frac{\Psi_0}{4}\bigl(e^{4 t}-1\bigr),\\
\|\partial_x \Theta\|_2^2 \le \Psi(t) &\le e^{4 t}\Psi_0.
\end{align*}
Consequently, the additional $\partial_x \Theta$–contributions in \eqref{lc:pre-gronwall} satisfy
\[
\frac{1}{c}\int_0^t\|\partial_x \Theta\|_2^2\,ds
+ \frac{1}{2c}\|\partial_x \Theta\|_2^2
\le
\frac{\Psi_0}{c4}\bigl(e^{4 t}-1\bigr)
+ \frac{\Psi_0}{2c}e^{4 t}
\le \frac{C}{c}\,\Psi_0\,e^{4 t},
\]
for some constant $C>0$ independent of $t$ and $c$. Since $\Psi_0\lesssim
\|\nabla^2 \Theta_0\|^2_2+\|\partial_x \Theta_0\|^2_2$, we can absorb these terms into the same type of bound as the last term in \eqref{lc:pre-gronwall}. Thus there exists a constant $C>0$, independent of $t$ and $c$, such that
\begin{align}
\label{lc:pre-gronwall_2}
    B \|w\|_2^2
    \leq{}& A \int_0^t \norm{w}^2_2 \,ds
    \\&\notag+C \qty(
    \frac{1}{c}+\frac{\nu}{c})\qty(\|\nabla^2 \Theta_0\|^2_2+\|\partial_x \Theta_0\|^2_2) e^{4 t}.
\end{align}

\paragraph{Final estimate.}
From \eqref{lc:A_B} recall that
\[
    B = \frac12 - \frac{1}{2c}, 
    \qquad
    A = \frac{\nu}{2c} + \frac{1}{2c},
\]
and therefore we have $B>0$ for $c>1$. Set
\begin{equation}
\label{lc:u_F}
u(t):=\|w(\cdot,t)\|_2^2,
\qquad
F_0:=
C \Bigl(
    \frac{1}{c}+\frac{\nu}{c}\Bigr)
    \Bigl(\|\nabla^2 \Theta_0\|^2_2+\|\partial_x \Theta_0\|^2_2\Bigr),
\end{equation}
so that \eqref{lc:pre-gronwall_2} can be written as
\[
B\,u(t) \le A \int_0^t u(s)\,ds + F_0e^{4 t},\qquad t\ge0.
\]
Dividing by $B>0$ gives
\[
u(t)\le \alpha(t)+\beta\int_0^t \,u(s)\,ds,\qquad t\ge0,
\]
with
\[
\alpha(t):=\frac{F_0}{B}e^{4 t},
\qquad
\beta:=\frac{A}{B}.
\]
Note that we have $\beta\ge0$, and $\alpha$ is non-decreasing in $t$. Hence we can apply the inhomogeneous Grönwall inequality to obtain
\[
u(t)\le \alpha(t)e^{\beta t}
= \frac{F_0}{B}
   \exp\Bigl(\bigl(4+\tfrac{A}{B}\bigr)t\Bigr),
\qquad t\ge0.
\]
By invoking the definitions of \eqref{lc:u_F} we conclude that
\begin{align}
\label{lc:post-gronwall}
\|w\|_2^2
\le{}& \frac{C}{B}
\Bigl(
    \frac{1}{c}+\frac{\nu}{c}\Bigr)
    \Bigl(\|\nabla^2 \Theta_0\|^2_2+\|\partial_x \Theta_0\|^2_2\Bigr)
    \exp\Bigl(\bigl(4+\tfrac{A}{B}\bigr)t\Bigr).
\end{align}
Lastly we note that $0<B< \frac{1}{2}$, $c>1$ and $\nu \ll 1$, therefore we have the bound
\[
4+\frac{A}{B}=4 +\frac1B \qty(\frac{\nu}{2c} + \frac{1}{2c} )\leq  C.
\]
Therefore we find
\begin{align}
\label{lc:post-opt}
\|w\|_2^2
\le{}& C
   \Bigl(
    \frac{1}{c}+\frac{\nu}{c}\Bigr)
    \Bigl(\|\nabla^2 \Theta_0\|^2_2+\|\partial_x \Theta_0\|^2_2\Bigr)
    e^{C t}.
\end{align}

\appendix
\section{Appendix}
\subsection{Mixing}
In this appendix we provide a detailed proof of \eqref{mix:Gy_est_pre} from Section \ref{sec:mix}. The notation follows the conventions of that section.
\label{mix:Appendix-A}
\paragraph{Proof of \eqref{mix:Gy_est_pre}.}
We follow the method of non-stationary phase as outlined in \cite{Stein,tao2007}.
We recall that on the integration domain of $I_\mathcal{D}$ we have the condition
\begin{align}
\label{mix:A:cos_bound}
  \Bigl|\cos\Bigl(\frac{ct}{2}-y\Bigr)\Bigr|\ge\delta,
\end{align}
and that
\[
  \partial_y\psi(y,t)
  = \frac{2}{c}\sin\Bigl(\frac{ct}{2}\Bigr)
                     \cos\Bigl(\frac{ct}{2}-y\Bigr).
\]
We now need to assume that $0<t\le T$ with $T\le \frac{\pi}{c}$, so that
$\frac{ct}{2}\in(0,\frac{\pi}{2}]$. On this interval we have the inequality
\[
  \sin z \;\ge\; \frac{2}{\pi}\,z,
  \qquad z\in\Bigl[0,\frac{\pi}{2}\Bigr],
\]
hence
\[
  \sin\Bigl(\frac{ct}{2}\Bigr)
  \;\ge\; \frac{2}{\pi}\,\frac{ct}{2}
  \;=\; \frac{c}{\pi}\,t.
\]
Therefore using \eqref{mix:A:cos_bound}, we have
\begin{align*}
  |\partial_y\psi(y,t)|
  \;=\; \frac{2}{c}\Bigl|\sin\Bigl(\frac{ct}{2}\Bigr)\Bigr|
       \Bigl|\cos\Bigl(\frac{ct}{2}-y\Bigr)\Bigr|
  \;\ge\; \frac{2}{c}\cdot \frac{c}{\pi}t\cdot\delta
  \;=\; \frac{2}{\pi}\,t\delta,
\end{align*}
so in particular
\begin{equation}\label{mix:A:dypsi-lower}
  \sup_{y \in \bigcup_j [y_j^-, y_j^+]}\frac{1}{|\partial_y\psi(y,t)|}
  \;\le\; \frac{\pi}{2}\,\frac{1}{t\delta}.
\end{equation}
We now recall \eqref{mix:IG_Y_IBP} as
\begin{align*}
   \sum_j &\int_{y_j^-}^{y_j^+} e^{-ik \psi(y,t)}\phi(y)\,dy\\
   &= \sum_j \int_{y_j^-}^{y_j^+}
        \frac{\phi(y)}{-ik\,\partial_y \psi(y,t)}\,
        \partial_y\bigl(e^{-ik \psi(y,t)}\bigr)\,dy \\
   &= \frac{1}{ik} \sum_j \left( \int_{y_j^-}^{y_j^+}
        e^{-ik \psi(y,t)}\,
        \partial_y\Bigl(\frac{\phi(y)}{\partial_y \psi(y,t)}\Bigr)\,dy-\left[\frac{\phi(y)e^{-ik \psi(y,t)}}{\partial_y \psi(y,t)}\right]_{y_j^-}^{y_j^+} \right),
\end{align*}
therefore
\begin{align*}
 &\abs{\sum_j \int_{y_j^-}^{y_j^+} e^{-ik \psi}\phi(y)\,dy}\\
&\quad \le \frac{1}{|k|}\sum_j \int_{y_j^-}^{y_j^+}
     \Bigl|\partial_y\Bigl(\frac{\phi(y)}{\partial_y\psi(y,t)}\Bigr)\Bigr|\,dy
     + \frac{1}{|k|}\left|\sum_j \left[\frac{\phi(y) e^{-ik \psi(y,t)}}{\partial_y \psi(y,t)}\right]_{y_j^-}^{y_j^+}\right|\\
 &\quad = \frac{1}{|k|}\sum_j \int_{y_j^-}^{y_j^+}
     \biggl|\frac{\partial_y\phi(y)}{\partial_y\psi(y,t)}
           -\phi(y)\,\frac{\partial_y^2\psi(y,t)}{(\partial_y\psi(y,t))^2}\biggr|\,dy
     \\&\quad\quad+ \frac{1}{|k|}\left|\sum_j \left[\frac{\phi(y) e^{-ik \psi(y,t)}}{\partial_y \psi(y,t)}\right]_{y_j^-}^{y_j^+}\right|\\
&\quad \le \frac{1}{|k|}\sum_j \int_{y_j^-}^{y_j^+}
     \biggl|\frac{\partial_y\phi(y)}{\partial_y\psi(y,t)}\biggr|\,dy
           \\&\quad\quad+\frac{1}{|k|}\sum_j \int_{y_j^-}^{y_j^+}\biggl|\phi(y)\,\frac{\partial_y^2\psi(y,t)}{(\partial_y\psi)^2}\biggr|\,dy
     + \frac{1}{|k|}\left|\sum_j \left[\frac{\phi(y) e^{-ik \psi(y,t)}}{\partial_y \psi}\right]_{y_j^-}^{y_j^+}\right|\\
 &\quad  =: \frac{1}{|k|}(J_1+J_2+J_3).
\end{align*}
\emph{Estimate of $J_1$.}
Using \eqref{mix:A:dypsi-lower} and the one–dimensional estimate
$\int_{\T}|\partial_y\phi|dy\le C\|\phi\|_{H^1(\T)}$, we obtain
\begin{align*}
  J_1=&\sum_j \int_{y_j^-}^{y_j^+}
     \biggl|\frac{\partial_y\phi(y)}{\partial_y\psi(y,t)}\biggr|\,dy
  \le \sup_{y \in \bigcup_j [y_j^-, y_j^+]}\frac{1}{|\partial_y\psi(y,t)|}
      \sum_j \int_{y_j^-}^{y_j^+}|\partial_y\phi(y)|dy
  \\\le& \frac{\pi}{2}\,\frac{1}{t\delta}
      \int_{\T}|\partial_y\phi(y)|dy
  \le C\,\frac{1}{t\delta}\,\|\phi(y)\|_{H^1(\T)},
\end{align*}
with an absolute constant $C>0$. Since $\|\eta\|_{H^1_y}=1$ 
we have
\[
  J_1
  \le C\,\frac{1}{t\delta}\,\|\widehat{\Theta}_0(k,\cdot)\|_{H^1_y}.
\]
\emph{Estimate of $J_2$.}
Note that
\[
  J_2=\sum_j \int_{y_j^-}^{y_j^+}\biggl|\phi(y)\,\frac{\partial_y^2\psi(y,t)}{(\partial_y\psi(y,t))^2}\biggr|\,dy
  \le \|\phi(y)\|_{L^\infty}
      \sum_j \int_{y_j^-}^{y_j^+}\frac{|\partial_y^2\psi(y,t)|}{|\partial_y\psi(y,t)|^2}\,dy
      \]
Hence we need to estimate the integral on the right-hand side. From the explicit formulas
\[
  \partial_y\psi(y,t)
  = \frac{2}{c}\sin\Bigl(\frac{ct}{2}\Bigr)
                     \cos\Bigl(\frac{ct}{2}-y\Bigr),\qquad
  \partial_y^2\psi(y,t)
  = \frac{2}{c}\sin\Bigl(\frac{ct}{2}\Bigr)
                      \sin\Bigl(\frac{ct}{2}-y\Bigr),
\]
we compute
\[
  \frac{|\partial_y^2\psi(y,t)|}{|\partial_y\psi(y,t)|^2}
  \leq \frac{c}{2\bigl|\sin(\frac{ct}{2})\bigr|}
    \frac{1}
         {\cos^2\Bigl(\frac{ct}{2}-y\Bigr)}.
\]
Using again the fact that $\sin(\frac{ct}{2})\ge \frac{c}{\pi}t$, we bound the integrals:
\begin{align*}
   \sum_j \int_{y_j^-}^{y_j^+}\frac{|\partial_y^2\psi(y,t)|}{|\partial_y\psi(y,t)|^2}\,dy
   &\le \frac{\pi}{2t}\sum_j \int_{y_j^-}^{y_j^+} \cos^{-2}\Bigl(\frac{ct}{2}-y\Bigr)\,dy
   \\&=  \frac{\pi}{2t}\sum_j \Bigl[-\tan\Bigl(\frac{ct}{2}-y\Bigr)\Bigr]_{y_j^-}^{y_j^+}.
\end{align*}
By definition, the boundaries $y_j^\pm$ satisfy $\bigl|\cos(\frac{ct}{2}-y_j^\pm)\bigr| = \delta$. Consequently, at any endpoint, the primitive is bounded by
\[
  \Bigl|\tan\Bigl(\frac{ct}{2}-y_j^\pm\Bigr)\Bigr| = \frac{\bigl|\sin(\frac{ct}{2}-y_j^\pm)\bigr|}{\bigl|\cos(\frac{ct}{2}-y_j^\pm)\bigr|} \le \frac{1}{\delta}.
\]
We conclude
\[
  \sum_j \int_{y_j^-}^{y_j^+}\frac{|\partial_y^2\psi(y,t)|}{|\partial_y\psi(y,t)|^2}\,dy
  \le \frac{C}{t\delta}.
\]
Then using $H^1(\T)\hookrightarrow L^\infty(\T)$, we get
\[
  J_2
  \le \|\phi\|_{L^\infty}
      \sum_j \int_{y_j^-}^{y_j^+}\frac{|\partial_y^2\psi(y,t)|}{|\partial_y\psi(y,t)|^2}\,dy
  \le C\,\frac{1}{t\delta}\,\|\widehat{\Theta}_0(k,\cdot)\|_{H^1_y}.
\]

\emph{Estimate of $J_3$.}
We simply evaluate the boundary terms explicitly:
\[
  J_3=\left|\sum_j \left[\frac{\phi(y)e^{-ik \psi(y,t)}}{\partial_y \psi(y,t)}\right]_{y_j^-}^{y_j^+}\right|
  \;\le\; \sum_{y\in \{y_j^\pm\}} \left|\frac{e^{-ik \psi(y,t)}\phi(y)}{\partial_y \psi(y,t)}\right|
  \;\le\; \sum_{y\in \{y_j^\pm\}} \frac{|\phi(y)|}{|\partial_y \psi(y,t)|}.
\]
Using the bound \eqref{mix:A:dypsi-lower} on the denominator and a $L^\infty$-bound on the numerator:
\[
  J_3
  \le  \frac{C}{t\delta} \|\phi\|_{L^\infty}.
\]
Using the embedding $H^1(\T)\hookrightarrow L^\infty(\T)$, we conclude
\[
  J_3 \le C\,\frac{1}{t\delta}\,\|\widehat{\Theta}_0(k,\cdot)\|_{H^1_y}.
\]

Combining the bounds for $J_1$, $J_2$ and $J_3$ gives
\begin{align}\label{mix:A:final}
  \abs{\sum_j \int_{y_j^-}^{y_j^+} e^{-ik \psi(y,t)}\phi(y)\,dy}
  \le \frac{1}{|k|}(J_1+J_2+J_3)
  \le \frac{C}{|k|t\delta}\,\|\widehat{\Theta}_0(k,\cdot)\|_{H^1_y},
\end{align}
for all $k\in\mathbb Z\setminus\{0\}$ and all $t\in(0,T]$ with
$T\le \frac{\pi}{c}$. The constant $C>0$ is absolute and does not depend
on $t$, $c$, $k$ or $T$, which is precisely \eqref{mix:Gy_est_pre}.
\subsection{Hypocoercivity}
This appendix collects technical details and auxiliary lemmas for Section \ref{sec:hyp}. The notation follows the conventions of that section.
\subsubsection*{Auxiliary Lemmas}
\begin{lemma}
	\label{hyp:L:energy}
	Let $\theta(k,y,\tau)$ be a sufficiently smooth solution of \eqref{hyp:eq}. Then, the following identities hold:
\begin{enumerate}
		\item $\displaystyle \frac{1}{2}\dv{}{\tau}E_0 = -\mu E_1,$
		\item $\displaystyle\frac{1}{2}\dv{}{\tau}E_1 = -\mu E_2 - E_3,$
		\item $\displaystyle\dv{}{\tau}E_3 = -E_4 - 2\mu\Re\inner{i\partial_y v\partial_y \theta}{\partial_y^2 \theta} - \mu\Re\inner{i\partial_y^2 v\theta }{\partial_y^2 \theta}-\varsigma\Re \inner{i\partial_{y}^2 v \theta}{\partial_{y} \theta},$
		\item $\displaystyle\frac{1}{2}\dv{}{\tau}E_4 = -\mu \norm{\partial_y v\partial_y \theta}_2^2 - 2\mu \Re \inner{\partial_y v\partial_y^2 v \theta}{\partial_y \theta}+\varsigma E_6,$
        \item $\displaystyle\dv{}{\tau} E_6= -\varsigma E_4+\varsigma E_7 -4 \mu E_6-2 \mu \inner{v \partial_y \theta}{\partial_y v \partial_y \theta},  $
        \item $\displaystyle \frac{1}{2}\dv{}{\tau}E_7=-\varsigma E_6-\mu \norm{v \partial_y \theta}_2^{2}+\mu E_4 -\mu E_7$.
	\end{enumerate}
\end{lemma}
\begin{proof}
Since the first four identities follow from standard energy estimates (see, e.g., \cite{Bedrossian2017,cotizelatigallay}), we only sketch their derivation. The last two balances involve the non-standard terms of the functional and rely more heavily on specific properties of the flow~$v$, so we provide additional details.

	In this proof we will repeatedly use the antisymmetric property of the advection terms under the $L^2 (\mathbb{T})$ inner product as
	\begin{equation}
\label{hyp:anitsymmetricprop}
\begin{aligned}
  \Re \inner{iv(y) f}{f}&=0, \quad\forall f\in L^2 (\mathbb{T}), \\
  \Re \inner{\partial_y f}{f}&=0, \quad\forall f\in L^2 (\mathbb{T}).
\end{aligned}
\end{equation}
    Further note that for our specific flow we have the property
    \begin{align}
    \label{hyp:A:sin_prop}
        \partial_y^2 v=-v.
    \end{align}
	We also note that boundary terms vanish due to periodicity of the domain. The first identity can be obtained by testing equation \eqref{hyp:eq} with $\theta$ and integrating over $\mathbb{T}$, we have
	\begin{align*}
		\frac{1}{2} \frac{d}{d t} E_0 &= \Re \langle \partial_\tau \theta, \theta \rangle_y = \Re \langle \mu \partial_y^2 \theta - i v \theta+ \varsigma\partial_y \theta, \theta \rangle_y = - \mu E_1,
	\end{align*}
	where we have used \eqref{hyp:anitsymmetricprop}. The second identity follows similarly by testing equation \eqref{hyp:eq} with $-\partial_y^2\theta$, which gives
	\begin{align*}
		\frac{1}{2} \frac{d}{d t} E_1 &= \Re \langle \partial_\tau \partial_y \theta, \partial_y \theta \rangle_y = \Re \big\langle \big[ \mu \partial_y^3 \theta - i  v \partial_y \theta - i  \partial_y v \theta +\varsigma \partial_y^2 \theta\big], \partial_y \theta \big\rangle_y \\
		&= - \mu E_2 -  E_3. 
	\end{align*}
	For the third identity we obtain
	\begin{align*}
		\dv{}{\tau}E_3 =& \Re\langle i\partial_y v \partial_\tau \theta, \partial_y \theta\rangle_y + \Re\langle i\partial_y v\theta, \partial_y \partial_\tau \theta\rangle_y \\
		=&\Re [\inner{v\partial_y v \theta}{\partial_y \theta}-\inner{\partial_y v \theta}{\partial_y v \theta} -\inner{\partial_y v \theta}{v \partial_y \theta}]\\
		&+\mu \Re [\inner{i \partial_y v \partial_y^2\theta}{\partial_y \theta}+ \inner{i  \partial_y v \theta}{\partial_y^3 \theta} ]+\Re \varsigma[\inner{i\partial_y v \partial_y \theta}{\partial_y \theta}+\inner{i\partial_y v \theta}{\partial_y^2 \theta}] \\
		=&  -E_4 - \mu \Re \langle i  \partial_y^2 v \theta, \partial_y^2 \theta \rangle_y - 2 \mu \Re \langle i  \partial_y v \partial_y \theta, \partial_y^2 \theta \rangle_y -\varsigma\Re\inner{i\partial_{y}^2 v \theta}{\partial_y  \theta}.
	\end{align*}
	The fourth identity is then obtained as follows
\begin{align*}
    \frac{1}{2}\dv{}{\tau}E_4 
    &= \Re \langle \partial_y v \partial_\tau \theta , \partial_y v \theta \rangle_y 
    = \Re \big\langle \partial_y v \big[ \mu \partial_y^2 \theta - i  v \theta +\varsigma\partial_y \theta \big], \partial_y v \theta \big\rangle_y \\
    &= -\mu \norm{\partial_y v\,\partial_y \theta}_2^{2}
    - 2\mu \Re \inner{\partial_y v\,\partial_y^2 v \,\theta}{\partial_y \theta}
    + \varsigma \Re \inner{\partial_y \theta}{(\partial_y v)^2 \theta},
\end{align*}
where we have integrated the $\mu$-term by parts and used property \eqref{hyp:anitsymmetricprop} for the $iv\theta$-term. For the $\varsigma$-term, we integrate by parts to find
\begin{equation*}
    \varsigma \Re \inner{\partial_y \theta}{(\partial_y v)^2 \theta} 
    = -\frac{\varsigma}{2}\Re \inner{\theta}{\partial_y\big((\partial_y v)^2\big)\theta}
    = -\varsigma \Re \inner{\theta}{(\partial_y v)(\partial_y^2 v)\theta}.
\end{equation*}
Finally, using property \eqref{hyp:A:sin_prop} then yields
\begin{align*}
    \frac{1}{2}\dv{}{\tau}E_4 
    &=-\mu \norm{\partial_y v\,\partial_y \theta}^2_2
    -2\mu \Re \inner{\partial_y v\,\partial_y^2 v \,\theta}{\partial_y \theta}
    +\varsigma E_6.
\end{align*}
    For the fifth identity, we use \eqref{hyp:eq} to compute
    \begin{align*}
        \dv{}{\tau}E_6
        &= \Re\inner{v\partial_y v\partial_\tau\theta}{\theta} +\Re\inner{v\partial_y v\theta}{\partial_\tau\theta}= 2\Re\inner{v\partial_y v\partial_\tau\theta}{\theta} \\
        &= 2\Re\inner{v\partial_y v\big(\mu\partial_y^2\theta - i v\theta + \varsigma\partial_y\theta\big)}{\theta} \\
        &= 2\mu\Re\inner{v\partial_y v\partial_y^2\theta}{\theta}
           + 2\varsigma\Re\inner{v\partial_y v\partial_y\theta}{\theta},
    \end{align*}
    where the $iv\theta$-term vanishes by \eqref{hyp:anitsymmetricprop}. 
    Integrating by parts one gets for the $\mu$-term
   \begin{align*}
    2\mu\,\Re\inner{v\partial_y v\,\partial_y^2\theta}{\theta}
    = -2\mu\,\inner{v\partial_y v\,\partial_y\theta}{\partial_y\theta}
       + \mu\,\inner{(v\partial_y v)''\theta}{\theta}.
       \end{align*}
       Now noting that by \eqref{hyp:A:sin_prop} we have
       \[
\partial_y^2(v\partial_y v) = \partial_y\big((\partial_y v)^2 + v\,\partial_y^2 v\big)
= \partial_y\big((\partial_y v)^2 - v^2\big)
= 2\,\partial_y v\,\partial_y^2 v - 2v\,\partial_y v
= -4\,v\,\partial_y v,
\] we get
\begin{align*}
    2\mu\,\Re\inner{v\partial_y v\,\partial_y^2\theta}{\theta}
    &= -2\mu\,\inner{v\partial_y v\,\partial_y\theta}{\partial_y\theta}
       + \mu\,\inner{(v\partial_y v)''\theta}{\theta} \\
    &= -2\mu\,\Re\inner{v\partial_y\theta}{\partial_y v\,\partial_y\theta}
       -4\mu E_6.
\end{align*}
For the $\varsigma$-term we integrate by parts yet again and obtain
        \[
        2\varsigma\,\Re\inner{v\,\partial_y v\,\partial_y\theta}{\theta}
        = -\varsigma\,\inner{\partial_y(v\,\partial_y v)\theta}{\theta}
        = -\varsigma\,\inner{((\partial_y v)^2 - v^2)\theta}{\theta}
        = -\varsigma E_4 + \varsigma E_7.
    \]
    Hence
    \[
        \dv{}{\tau}E_6
        = -\varsigma E_4 + \varsigma E_7
          -4\mu E_6
          -2\mu\,\Re\inner{v\partial_y\theta}{\partial_y v\,\partial_y\theta}.
    \]
For the sixth identity, we use \eqref{hyp:eq}
to compute
    \begin{align*}
        \frac{1}{2}\dv{}{\tau}E_7
        &= \Re\inner{v\partial_\tau\theta}{v\theta} \\
        &= \Re\inner{v\big(\mu\partial_y^2\theta - i v\theta + \varsigma\partial_y\theta\big)}{v\theta} \\
        &= \mu\,\Re\inner{v^2\partial_y^2\theta}{\theta}
           + \varsigma\Re\inner{v^2\partial_y\theta}{\theta},
    \end{align*}
    again, using \eqref{hyp:anitsymmetricprop} to discard the $iv\theta$–term. 
    Integrating by parts we obtain for the $\mu$-term
  \begin{align*}
    \mu\,\Re\inner{v^2\partial_y^2\theta}{\theta}
    &= -\mu\,\inner{v^2\partial_y\theta}{\partial_y\theta}
       - \mu\,\Re\inner{\partial_y(v^2)\partial_y\theta}{\theta}.
\end{align*}
Using an additional integration by parts on the last term, we obtain
\[
    \Re\inner{\partial_y(v^2)\partial_y\theta}{\theta}
    = -\frac12\,\inner{\partial_y^2(v^2)\theta}{\theta}.
\]
Hence, noting that via \eqref{hyp:A:sin_prop}
    \[
\partial_y^2(v^2) = \partial_y\big(2v\,\partial_y v\big)
       = 2\big[(\partial_y v)^2 + v\,\partial_y^2 v\big]
       = 2\big((\partial_y v)^2 - v^2\big),
\]
we get
\begin{align*}
    \mu\,\Re\inner{v^2\partial_y^2\theta}{\theta}
    &= -\mu\,\inner{v^2\partial_y\theta}{\partial_y\theta}
       + \frac{\mu}{2}\,\inner{(v^2)''\theta}{\theta} \\
    &= -\mu\|v\partial_y\theta\|_2^{2}
       + \mu\,\inner{\big((\partial_y v)^2 - v^2\big)\theta}{\theta} \\
    &= -\mu\|v\partial_y\theta\|_2^{2} + \mu E_4 - \mu E_7.
\end{align*}
An integration by parts of the $\varsigma$-term gives
        \[
        \varsigma\,\Re\inner{v^2\partial_y\theta}{\theta}
        = -\frac{c}{2}\,\inner{\partial_y(v^2)\theta}{\theta}
        = -\varsigma\,\inner{v\,\partial_y v\,\theta}{\theta}
        = -\varsigma E_6.
    \]

    Therefore
    \[
        \frac{1}{2}\dv{}{\tau}E_7
        = -\varsigma E_6 -\mu\|v\partial_y\theta\|_2^{2}
          +\mu E_4 -\mu E_7.
    \]

\end{proof}
\begin{lemma}\cite[Proposition 2.7]{Bedrossian2017}
	\label{hyp:L:spectral_gap}
	There exists a constant $C_{s} \geq 1$ such that for all $\sigma \in (0,1]$,
	\begin{equation}
		\sigma^{\tfrac{1}{2}} E_0 \lesssim C_{s} [\sigma E_1 + E_4]\,.
	\end{equation}
\end{lemma}
\subsubsection*{Detailed optimisation}
\label{hyp:A:optimisation}
We recall that in order to apply coercivity \eqref{hyp:coercivity}, we need an inequality of the form
\begin{align}
\label{hyp:ineq_needed}
    \dv{}{\tau}\Phi+\frac{1}{8}\lambda_\mu\qty(4E_0+5\alpha_0  E_1+6\gamma_0, E_4+5\gamma_1 E_7)\leq 0
\end{align}
where $\lambda_\mu$ is the rate that will be determined below. Form \eqref{hyp:param_choice} we have
\begin{align*}
\alpha_0=\beta_0^{1/2} \mu^{p},\quad\gamma_0=16\frac{\beta_0^{3/2}}{\mu^{p}},\quad \gamma_1=\frac{\beta_0^{1/2}}{\mu^{q}}.
\end{align*}
We perform the optimisation procedure in two steps, we first address the critical scaling in $\mu$ and then tune the scaling in $\beta_0$ and the constants. 

\paragraph{$\mu$-Optimisation:}
arTo streamline the presentation, we temporarily suppress all constants (including $\beta_0$) and extract from
\eqref{hyp:post_spectral} an optimisation problem in~$\mu$.
Ignoring prefactors, \eqref{hyp:post_spectral} has the form
\begin{align}
\label{hyp:post_spectral_mu_free}
    \dv{}{\tau}\Phi
    + \mu^{\frac{1-q}{2}} E_0
    + \mu^{1-q} E_1
    + E_4
    + \mu^{\frac{2\ell-p-q}{2}} E_7
    \le 0 .
\end{align}
To recover the $\mu$--weights appearing in \eqref{hyp:ineq_needed} via \eqref{hyp:param_choice},
we aim to factor out a single rate $\lambda_\mu \sim \mu^{\frac{1-q}{2}}$ and bound
\begin{align*}
    \mu^{\frac{1-q}{2}}
    \Bigl(E_0+\mu^{p}E_1+\mu^{-p}E_4+\mu^{-q}E_7\Bigr)
    \;\le\;
    \mu^{\frac{1-q}{2}}E_0+\mu^{1-q}E_1+E_4+\mu^{\frac{2\ell-p-q}{2}}E_7 .
\end{align*}
Termwise, this requires
\begin{align*}
    \mu^{\frac{1-q}{2}+p}\le \mu^{1-q},\qquad
    \mu^{\frac{1-q}{2}-p}\le 1,\qquad
    \mu^{\frac{1-q}{2}-q}\le \mu^{\frac{2\ell-p-q}{2}} .
\end{align*}
In addition, we must enforce the $\mu$--constraints introduced during the proof,
namely \eqref{hyp:mu:E_6_condition}, \eqref{hyp:mu:E_1_condition}, and \eqref{hyp:mu:E_0_condition}:
\begin{align*}
    \mu^{2(\ell-p)}\le \mu^{\frac{2\ell-p-q}{2}},\qquad
    \mu^{\frac{2\ell+p+q}{2}}\le \mu^{1-q},\qquad
    \mu^{\frac{1-q}{2}}\ge \mu^{1-p}.
\end{align*}
Since $0<\mu<1$, we use $\mu^{A}\le \mu^{B}\iff A\ge B$ and $\mu^{A}\ge \mu^{B}\iff A\le B$
to rewrite these as constraints on exponents:
\[
\begin{alignedat}{2}
    \text{(A)}\;& 2\ell-p+2q\le 1,
    \qquad&
    \text{(D)}\;& 2\ell+q\ge 3p,\\[2pt]
    \text{(B)}\;& p\le \frac{1-q}{2},
    \qquad&
    \text{(E)}\;& 2\ell+p+3q\ge 2,\\[2pt]
    \text{(C)}\;& \frac{1-q}{2}\le p,
    \qquad&
    \text{(F)}\;& 2p\le 1+q.
\end{alignedat}
\]
Moreover, we retain the order assumptions used in the proof (cf.\ \eqref{hyp:p_q_fix} and \eqref{hyp:pq_y_fix}),
\[
\text{(O)}\qquad 0<q<p\le \ell<1 .
\]

\medskip
\noindent\emph{Step 1: Fixing $p$.}
From the two-sided bound (B)–(C) we obtain
\begin{equation}
\label{hyp:mu_opt_p}
    p=\frac{1-q}{2}.
\end{equation}

\noindent\emph{Step 2: Fixing $\ell$.}
Insert \eqref{hyp:mu_opt_p} into (A) and (D):
\begin{align*}
\text{(A)}\;&\ 2\ell-\frac{1-q}{2}+2q\le 1
\ \Longleftrightarrow\ 4\ell+5q\le 3,\\[0.3em]
\text{(D)}\;&\ 2\ell+q\ge 3\frac{1-q}{2}
\ \Longleftrightarrow\ 4\ell+5q\ge 3.
\end{align*}
Hence $4\ell+5q=3$, i.e.
\begin{equation}
\label{hyp:mu_opt_ell}
    \ell=\frac{3-5q}{4}.
\end{equation}

\noindent\emph{Step 3: Admissible $q$--range.}
Using (O) with \eqref{hyp:mu_opt_p} gives
\[
q<p=\frac{1-q}{2}
\ \Longleftrightarrow\
3q<1
\ \Longleftrightarrow\
q\in\Bigl(0,\frac13\Bigr).
\]
For such $q$, \eqref{hyp:mu_opt_ell} yields $\ell\in(1/3,3/4)$ and also $p\le \ell$.
Finally, with \eqref{hyp:mu_opt_p}--\eqref{hyp:mu_opt_ell}, the remaining constraints (E) and (F)
reduce to identities/redundancies:
\[
2\ell+p+3q=2,
\qquad
2p\le 1+q \ \Longleftrightarrow\ q\ge 0.
\]
Therefore the feasible set is the segment
\[
q\in\Bigl(0,\frac13\Bigr),\qquad
p=\frac{1-q}{2},\qquad
\ell=\frac{3-5q}{4}.
\]
Eliminating $q$ gives the parametrisation in terms of~$\ell$:
\begin{equation}
\label{hyp:y_final_range}
\ell\in\Bigl(\tfrac13,\tfrac34\Bigr),\qquad
p=\frac{1+2\ell}{5},\qquad
q=\frac{3-4\ell}{5}.
\end{equation}

With \eqref{hyp:y_final_range}, we have $\frac{1-q}{2}=p$, so \eqref{hyp:post_spectral} can be rewritten as
\begin{align}
\label{hyp:post_mu}
    \dv{}{\tau}\Phi
    + \mu^p\Bigl[
      \frac{\beta_0^{3/2}}{2C_{s}}\,E_0
      + \mu^{p}\beta_0^{3/2}\,E_1
      + \mu^{-p}\frac{\beta_0}{16}\,E_4
      + \frac14\,\mu^{-q}\beta_0\varsigma_k\,E_7
    \Bigr]
    \le 0 .
\end{align}

\paragraph{$\beta_0$-Optimisation:}
Optimising in $\beta_0$ and the constants is now straight forward. We first consider only the bracket of \eqref{hyp:post_mu}. We need to move the factor of $\beta_0^{5/4}$, in front of $E_0$, out of the bracket.
\begin{align*}
    &\mu^{p}\qty[ \frac{\beta_0^{5/4}}{2C_{s}}E_0 + \mu^p \beta_0^{3/2}E_1+ \mu^{-p}\frac{\beta_0 }{16} E_4 + \frac{1}{4}\mu^{-q}\beta_0\varsigma_k E_7]\\
    \geq&\beta_0^{5/4}\mu^{p}\qty[ \frac{1}{2C_{s}}E_0 + \beta_0^{1/4}\mu^p E_1+ \frac{1 }{16\beta_0^{1/4}} \mu^{-p}E_4 + \frac{1}{4\beta_0^{1/4} }\mu^{-q}\varsigma_k E_7].
 \end{align*}
Now recall that $\beta_0$ is small (see \eqref{hyp:beta_0_condition_spectral}) and that $C_{s}\geq 1$, hence we have
\begin{align*}
    &\beta_0^{5/4}\mu^{p}\qty[ \frac{1}{2C_{s}}E_0 + \beta_0^{1/4}\mu^p E_1+ \frac{1 }{16\beta_0^{1/4}} \mu^{-p}E_4 + \frac{1}{4\beta_0^{1/4} }\mu^{-q}\varsigma_k E_7]\\
        \geq{}& \frac{1}{2C_{s}}\beta_0^{5/4}\mu^{p}\qty[ E_0 + \beta_0^{1/2}\mu^p E_1+ 16\beta_0^{3/2} \mu^{-p}E_4 +\beta_0^{1/2}\mu^{-q}\varsigma_k E_7]\\
          ={}&\frac{1}{2C_{s}}\beta_0^{5/4}\mu^{p}\qty[ E_0 + \alpha_0 E_1+ \gamma_0 E_4 +\varsigma_k\gamma_1 E_7]\\
    \geq{}&\frac{1}{2C_{s}}\beta_0^{5/4}\varsigma_k\mu^{p}\qty[ E_0 + \alpha_0 E_1+ \gamma_0 E_4 +\gamma_1 E_7]\\
    \geq{}&\frac{1}{12C_{s}}\beta_0^{5/4}\varsigma_k\mu^{p}\qty[ 4E_0 + 5\alpha_0 E_1+ 6\gamma_0 E_4 +5\gamma_1 E_7],
\end{align*}
where in the penultimate step we used $\varsigma_k\leq \varsigma_1$ and absorbed a factor of $\max(\varsigma_1,1)$ into $C_s$. Therefore, as \eqref{hyp:post_mu} holds, we obtain
\begin{align}\label{hyp:post_opt}
    &\dv{}{\tau}\Phi+ 
    \frac{\beta_0^{5/4}}{12C_{s}}\varsigma_k\mu^{p}\qty[ 4E_0 + 5\alpha_0 E_1+ 6\gamma_0 E_4 +5\gamma_1 E_7]\leq 0.
 \end{align}

  \bibliographystyle{abbrv}
\bibliography{bibliography}

@article{cotizelati2024,
  title={Mixing in incompressible flows: transport, dissipation, and their interplay},
  author={Coti Zelati, M. and Crippa, G. and Iyer, G. and Mazzucato, A. L.},
  journal={Notices of the American Mathematical Society},
  volume={71},
  number={5},
  pages={593--604},
  year={2024},
  doi = {10.1090/noti2929}
}

@article{pierrehumbert1994tracer,
  title={Tracer microstructure in the large-eddy dominated regime},
  author={Pierrehumbert, R. T.},
  journal={Chaos, Solitons \& Fractals},
  volume={4},
  number={6},
  pages={1091--1110},
doi={10.1016/0960-0779(94)90139-2},
  year={1994},
  publisher={Elsevier}
}

@article{thiffeault2012,
  title={Using multiscale norms to quantify mixing and transport},
  author={Thiffeault, J.-L.},
  journal={Nonlinearity},
  volume={25},
  number={2},
  pages={R1},
doi= {10.1088/0951-7715/25/2/R1},
  year={2012},
  publisher={IOP Publishing}
}

@ARTICLE{Bedrossian2017,
    title={Enhanced Dissipation, Hypoellipticity, and Anomalous Small Noise Inviscid Limits in Shear Flows},year={2017},author={J. Bedrossian and M. {Coti Zelati}},doi={10.1007/s00205-017-1099-y},pmid={null},pmcid={null},mag_id={2963856888},journal={Archive for Rational Mechanics and Analysis},volume={224}
}

@article{Coble2024,
  author = {Coble, D. and He, S.},
  title = {A note on enhanced dissipation of time-dependent shear flows},
  journal = {Communications in Mathematical Sciences},
  volume = {22},
  number = {6},
  pages = {1685--1700},
  year = {2024},
  month = {7},
  doi = {10.4310/CMS.2024.v22.n6.a10},
  publisher = {International Press of Boston, Inc.},
}

@article{cotizelatigallay,
author = {Coti Zelati, M. and Gallay, T.},
title = {Enhanced dissipation and Taylor dispersion in higher-dimensional parallel shear flows},
journal = {Journal of the London Mathematical Society},
volume = {108},
number = {4},
pages = {1358-1392},
doi = {10.1112/jlms.12782},
year = {2023}
}

@article{cotizelati2019,
  title={On the relation between enhanced dissipation timescales and mixing rates},
  author={Coti Zelati, M. and Delgadino, M. G. and Elgindi, T. M.},
  journal={Communications on Pure and Applied Mathematics},
  volume={73},
  number={6},
  pages={1205--1244},
doi = {10.1002/cpa.21831},
  year={2020},
  publisher={Wiley Online Library}
}

@Book{villani,
  author    = {Villani, C.},
  publisher = {American Mathematical Society},
  title     = {Hypocoercivity},
  year      = {2009},
  number    = {950},
  series    = {Memoirs of the American Mathematical Society},
doi = {10.1090/S0065-9266-09-00567-5},
}

@ARTICLE{Wei2019,
  author    = {Wei, D. and Zhang, Z.},
  title     = {Enhanced dissipation for the Kolmogorov flow via the hypocoercivity method},
  journal   = {Science China Mathematics},
  volume    = {62},
  year      = {2019},
  pages     = {1219--1232},
  doi       = {10.1007/s11425-018-9508-5},
}

@book{Stein,
 ISBN = {9780691032160},
 author = {E. M. Stein},
 publisher = {Princeton University Press},
 title = {Harmonic Analysis (PMS-43): Real-Variable Methods, Orthogonality, and Oscillatory Integrals. (PMS-43)},
 urldate = {2024-10-30},
 year = {1993},
doi={10.1515/9781400883929},
}

@Article{Zelati2021,
  author   = {Coti Zelati, M. and Drivas, T. D.},
  journal  = {Annali della Scuola Normale Superiore di Pisa. Classe di Scienze. Serie V},
  title    = {A stochastic approach to enhanced diffusion},
  year     = {2021},
  issn     = {0391-173X},
  number   = {2},
  pages    = {811--834},
  volume   = {22},
  doi      = {10.2422/2036-2145.201911_013},
  language = {English},
  zbl      = {1491.35351},
  zbmath   = {7417788},
}

@article{elgindi2025,
  title={Optimal enhanced dissipation and mixing for a time-periodic, Lipschitz velocity field on $\mathbb{T}^2$},
  author={Elgindi, T. M. and Liss, K. and Mattingly, J. C.},
  journal={Duke Mathematical Journal},
  volume={174},
  number={7},
  pages={1209--1260},
  year={2025},
  doi = {10.1215/00127094-2024-0057}
}

@article{Bedrossian2021,
  author    = {Bedrossian, J. and Blumenthal, A. and Punshon-Smith, S.},
  title     = {Almost-sure enhanced dissipation and uniform-in-diffusivity exponential mixing for advection-diffusion by stochastic Navier-Stokes},
  journal   = {Probability Theory and Related Fields},
  volume    = {179},
  number    = {3--4},
  pages     = {777--834},
  year      = {2021},
  doi       = {10.1007/s00440-020-01010-8},
  mrnumber  = {4242626}
}

@article{cooperman2025,
  title={A Harris theorem for enhanced dissipation, and an example of Pierrehumbert},
  author={Cooperman, W. and Iyer, G. and Son, S.},
  journal={arXiv preprint arXiv:2403.19858},
  year={2025},
  doi = {10.48550/arXiv.2403.19858}
}

@article{Gess2025,
  author    = {Gess, B. and Yaroslavtsev, I.},
  title     = {Stabilization by transport noise and enhanced dissipation in the Kraichnan model},
  journal   = {Journal of Evolution Equations},
  volume    = {25},
  pages     = {42},
  year      = {2025},
  doi       = {10.1007/s00028-025-01066-w}
}

@article{Seis2025,
  title={Exponential mixing by random cellular flows},
  author={Navarro-Fernández, V. and Seis, C.},
  journal={arXiv preprint arXiv:2502.17273},
  year={2025},
  doi = {10.48550/arXiv.2502.17273}
}

@article{mazzucato2026,
  title={Enhanced dissipation by advection and applications to PDEs},
  author={Mazzucato, A. L. and Feng, Y. and Nobili, C.},
  journal={Applied Mathematical Modelling},
  volume={150},
  pages={116310},
  year={2026},
  doi = {10.1016/j.apm.2025.116310}
}

@article{liu2004,
title={Strange eigenmodes and decay of variance in the mixing of diffusive tracers},
author={Liu, W. and Haller, G.},
journal={Physica D: Nonlinear Phenomena},
volume={188},
number={1},
pages={1--39},
year={2004},
doi = {10.1016/S0167-2789(03)00287-2}
}

@article{vanneste2007,
title={Fast scalar decay in a shear flow: modes and pseudomodes},
author={Vanneste, J. and Byatt-Smith, J. G.},
journal={Journal of Fluid Mechanics},
volume={572},
pages={219--229},
year={2007},
doi = {10.1017/S0022112006003661}
}

@article{giona2004,
title={Universality and imaginary potentials in advection--diffusion in closed flows},
author={Giona, M. and Cerbelli, S. and Vitacolonna, V.},
journal={Journal of Fluid Mechanics},
volume={513},
pages={221--237},
year={2004},
doi = {10.1017/S002211200400984X}
}

@article{reddy1994,
title={Pseudospectra of the Convection-Diffusion Operator},
author={Reddy, S. C. and Trefethen, L. N.},
journal={SIAM Journal on Applied Mathematics},
volume={54},
number={6},
pages={1634--1649},
year={1994},
doi = {10.1137/S0036139993246982}
}

@article{benthaus2026,
  title={Enhanced dissipation via time-modulated velocity fields},
  author={Benthaus, J. and Nobili, C.},
  journal={Evolution Equations and Control Theory},
  volume={15},
  pages={21--50},
  year={2026},
  doi = {10.3934/eect.2025051}
}

@article{batchelor1959,
title={Small-scale variation of convected quantities like temperature in turbulent fluid: Part 1. General discussion and the case of small conductivity},
author={Batchelor, G. K.},
journal={Journal of Fluid Mechanics},
volume={5},
number={1},
pages={113--133},
year={1959},
doi = {10.1017/S002211205900009X}
}

@misc{tao2007,
  title={MATH 247B: Fourier analysis},
  author={Tao, T.},
  howpublished={Course notes, University of California, Los Angeles},
  year={2007},
  url = {https://www.math.ucla.edu/~tao/247b.1.07w/}
}

@article{he2025,
  title={Time-dependent flows and their applications in parabolic-parabolic Patlak-Keller-Segel systems Part I: Alternating flows},
  author={He, S.},
  journal={Journal of Functional Analysis},
  volume={288},
  number={5},
  pages={110786},
  year={2025},
  doi = {10.1016/j.jfa.2024.110786}
}
\end{document}